\documentclass[11pt,reqno]{amsart}

\usepackage[a4paper,left=35mm,right=35mm,top=30mm,bottom=30mm,marginpar=25mm]{geometry}

\usepackage{amsfonts}
\usepackage{amsmath}
\usepackage{amssymb}
\usepackage{amsthm}
\usepackage[latin1]{inputenc}
\usepackage{eurosym}
\usepackage[dvips]{graphics}
\usepackage{graphicx}
\usepackage{epsfig}
\usepackage{hyperref}
\usepackage{dsfont}
\usepackage{color}

\usepackage[displaymath,mathlines]{lineno}

\allowdisplaybreaks

\usepackage{hyperref}

\usepackage{ifthen}
\makeindex

\newcommand{\argmin}{\operatorname{argmin}}

\newcommand{\Rr}{{\mathbb{R}}}

\newcommand{\Nn}{{\mathbb{N}}}

\newcommand{\Tt}{{\mathbb{T}}}

\newcommand{\Hh}{{\overline{H}}}

\newcommand{\bx}{\boldsymbol{x}}

\newcommand{\by}{\boldsymbol{y}}

\newcommand{\ba}{\boldsymbol{a}}
\newcommand{\bb}{\boldsymbol{b}}

\newcommand{\bxi}{\boldsymbol{\xi}}
\newcommand{\bita}{\boldsymbol{\eta}}
\newcommand{\bc}{\boldsymbol{c}}
\newcommand{\bd}{\boldsymbol{d}}

\def\leq{\leqslant}
\def\geq{\geqslant}

\numberwithin{equation}{section}

\newtheoremstyle{thmlemcorr}{10pt}{10pt}{\itshape}{}{\bfseries}{.}{10pt}{{\thmname{#1}\thmnumber{
                        #2}\thmnote{ (#3)}}}
\newtheoremstyle{thmlemcorr*}{10pt}{10pt}{\itshape}{}{\bfseries}{.}\newline{{\thmname{#1}\thmnumber{
                        #2}\thmnote{ (#3)}}}
\newtheoremstyle{defi}{10pt}{10pt}{\itshape}{}{\bfseries}{.}{10pt}{{\thmname{#1}\thmnumber{
                        #2}\thmnote{ (#3)}}}
\newtheoremstyle{remexample}{10pt}{10pt}{}{}{\bfseries}{.}{10pt}{{\thmname{#1}\thmnumber{
                        #2}\thmnote{ (#3)}}}
\newtheoremstyle{ass}{10pt}{10pt}{}{}{\bfseries}{.}{10pt}{{\thmname{#1}\thmnumber{
                        A#2}\thmnote{ (#3)}}}

\theoremstyle{thmlemcorr}
\newtheorem{theorem}{Theorem}
\numberwithin{theorem}{section}
\newtheorem{lemma}[theorem]{Lemma}
\newtheorem{corollary}[theorem]{Corollary}
\newtheorem{proposition}[theorem]{Proposition}

\theoremstyle{thmlemcorr*}
\newtheorem{theorem*}{Theorem}
\newtheorem{lemma*}[theorem]{Lemma}
\newtheorem{corollary*}[theorem]{Corollary}
\newtheorem{proposition*}[theorem]{Proposition}
\newtheorem{problem*}[theorem]{Problem}
\newtheorem{conjecture*}[theorem]{Conjecture}

\theoremstyle{defi}

\theoremstyle{remexample}
\newtheorem{remark}[theorem]{Remark}

\theoremstyle{ass}

\begin{document}

\title[The short title]{One-dimensional, non-local, first-order, stationary mean-field games with congestion: a Fourier approach}

\author{Levon Nurbekyan}
\address[L. Nurbekyan]{
	King Abdullah University of Science and Technology (KAUST), CEMSE Division, Thuwal 23955-6900, Saudi Arabia. }
\email{levon.nurbekyan@kaust.edu.sa}


\thanks{
	L. Nurbekyan was partially supported by KAUST baseline and start-up funds and 
	KAUST SRI, Uncertainty Quantification Center in Computational Science and Engineering.
}
\date{\today}

\begin{abstract}
Here, we study a one-dimensional, non-local mean-field game model with congestion. When the kernel in the non-local coupling is a trigonometric polynomial we reduce the problem to a finite dimensional system. Furthermore, we treat the general case by approximating the kernel with trigonometric polynomials. Our technique is based on Fourier expansion methods.
\end{abstract}

\maketitle

\tableofcontents

\section{Introduction}

In this paper, we consider the following mean-field game (MFG) model
\begin{equation}\label{eq: main}
	\begin{cases}
	\frac{(u_x+c)^2}{2m^{2-\alpha}}+V(x)=\int\limits_{\Tt} G(x-y)m(y)dy+\Hh,\\
	(m^{\alpha-1}(u_x+c))_x=0,\\
	m>0,\ \int\limits_{\Tt} m(x)dx=1,
	\end{cases}
\end{equation}
where, $G\in C^2(\Tt)$ is a given kernel, $V\in C^2(\Tt)$ is a given $C^2$ potential and $0< \alpha\leq 2,\ c \in \Rr$ are given parameters. The unknowns are functions $u,m:\Tt \to \Rr$ and the number $\Hh \in \Rr$. We study the existence of smooth solutions for \eqref{eq: main} and analyze their properties and solution methods.  

MFGs theory was introduced by J-M. Lasry and P-L. Lions in \cite{ll1,ll2,ll3,ll4} and by M. Huang, P. Caines and R. Malham\'{e} in \cite{Caines2, Caines1} to study large populations of agents that play dynamic differential games. Mathematically, MFGs are given by the following system
\begin{equation}\label{eq: mfg}
\begin{cases}
-u_t(x,t)-\sigma^2 \Delta_x u(x,t)+H(x,D_xu(x,t),m(x,t),t)=0,\\
m_t(x,t)-\sigma^2 \Delta_x m(x,t)-\text{div}_x\left(D_p H(x,D_xu(x,t),m(x,t),t)\right)=0,\\
m(x,0)=m^0(x),\quad u(x,T)=u^T(x,m^0(x)). 
\end{cases}
\end{equation}
where $m(x,t)$ is the distribution of the population at time $t$, and $u(x,t)$ is the value function of the individual player, and $T$ is the terminal time. Furthermore, $H:\Tt^d \times \Rr^d \times X \times \Rr \to \Rr,\ (x,p,m,t)\mapsto H(x,p,m,t)$ is the Hamiltonian of the system, where $X=\Rr^+$ or $L^1(\Tt^d;\Rr^+)$ or $\Rr^+ \times L^1(\Tt^d;\Rr^+)$, and $\sigma\geq0$ is the diffusion parameter. Finally, $m^0,\ u^T$ are given initial-terminal conditions.

Suppose $L:\Tt^d \times \Rr^d \times X \times \Rr \to \Rr,\ (x,v,m,t)\mapsto L(x,v,m,t)$ is the Legendre transform of $H$. Then, formally, \eqref{eq: mfg} are the optimality conditions for a population of agents where each agent aims to minimize the action
\begin{equation}\label{eq: control problem}
u(x,t)=\inf\limits_{v}\{\mathbb{E}\int\limits_{t}^{T} L\big(x(s),v(s),m(x(s),s),s\big)ds+u^T\big(x(T),m(x(T),T)\big)\},
\end{equation}
where the infimum is taken over all progressively measurable controls $v(s)$, and trajectories $x(s)$ are governed by
\begin{equation*}
	dx(s)=v(s)ds+\sqrt{2}\sigma dW_s,\quad x(0)=x,
\end{equation*}
for a standard $d$-dimensional Brownian motion $\{W_s\}$. Assume that are driven by mutually independent Brownian motions.

Indeed, the first equation in \eqref{eq: mfg} is the Hamilton-Jacobi equation for the value function $u$. Furthermore, optimal velocities of agents are given by
\begin{equation*}
	v(t)=-D_p H(x,D_xu(x(t),t),m(x(t),t)),
\end{equation*}
thus the second equation in \eqref{eq: mfg} which is the corresponding Fokker-Planck equation. Rigorous derivations of \eqref{eq: mfg} in various contexts can be found in \cite{ll1,ll2,ll3,ll4,llg2,Carda,GS} and references therein.

Actions of the total population affect an individual agent through the dependence of $H$ and $L$ on $m$. The type of the dependence of $H$ and $L$ on $m$ is called the \textit{coupling}, and it can be either local, global or mixed. Spatial preferences of agents are encoded in the $x$ dependence of $H$ and $L$. 

Our problem of interest \eqref{eq: main} is the 1-dimensional, stationary, first-order version of \eqref{eq: mfg} with Hamiltonian
\begin{equation}\label{eq: Hamiltonian}
	H(x,p,m,t)=\frac{(p+c)^2}{2m^{2-\alpha}}+V(x)-\int\limits_{\Tt} G(x-y)m(y)dy.
\end{equation}

Since seminal papers \cite{ll1,ll2,ll3,ll4,Caines2,Caines1} a substantial amount of research has been done in MFGs. Classical solutions were studied extensively both in stationary and non-stationary settings in \cite{ll3, GPM1, GM, PV15} and in \cite{GPim1, GPim2, GPM2, GPM3}, respectively. Weak solutions were addressed in \cite{porretta, porretta2, Cd2, cgbt} for time-dependent problems and in \cite{FG2} for stationary problems. Numerical methods can be found in \cite{AP, achdou2013finite, achdouperez'12, CDY, DY, carlinisilva'14, camillisilva'12, gueant'12, AFG}.

Nevertheless, most of the previous work concerns problems where Hamiltonian does not have singularity at $m=0$. The problems where Hamiltonian has singularity at $m=0$, such as in \eqref{eq: Hamiltonian}, are called \textit{congestion} problems. The reason is that the Lagrangian corresponding to $H$ in \eqref{eq: Hamiltonian} is
\begin{equation*}
	L(x,v,m,t)=\frac{v^2m^{2-\alpha}}{2}+c v-V(x)+\int\limits_{\Tt}G(x-y)m(y)dy,
\end{equation*}
and in the view of \eqref{eq: control problem} agents pay high price for moving at high speeds in dense areas.

Congestion problems were previously studied in \cite{LCDF, GMit, GVrt2, Graber2, GNP2, Gueant2, meszaros'15, meszaros'16}. Uniqueness of smooth solutions was established in \cite{LCDF}. Existence of smooth solutions for stationary second-order local MFG with quadratic Hamiltonian was established in \cite{GMit}. Short-time existence and uniqueness of smooth and weak solutions for time-dependent second-order local MFGs were addressed in \cite{GVrt2} and \cite{Graber2}, respectively. Analysis of stationary first-order local MFGs in 1-dimensional setting is performed in \cite{GNP2}. Problems on graphs are considered in \cite{Gueant2}. MFG models with density constraints (hard congestion) and local coupling are addressed in \cite{meszaros'15} (second-order case) and \cite{meszaros'16} (first-order case). To our knowledge, existence of smooth solutions for stationary first-order MFGs with global coupling has not been studied before.

One of the main tools of analysis in MFGs theory is the method of a priori estimates. See \cite{GPV, Carda} and references therein for a detailed account on a priori-estimates methods in MFGs.

Here, we take a different route. Firstly, using the 1-dimensional structure of the problem, we reduce it to an equation with only $m$ and $\Hh$ as unknowns. Indeed, from the second equation in \eqref{eq: main} we have that
\begin{equation}\label{eq: reduction}
u_x(x)+c=\frac{j}{m(x)^{\alpha-1}},\quad \forall x\in \Tt,
\end{equation}
where $j$ is some constant that we call \textit{current}. Therefore, \eqref{eq: main} can be written in an equivalent form
\begin{equation}\label{eq: main current}
\begin{cases}
\frac{j^2}{2m^{\alpha}(x)}+V(x)=\int\limits_{\Tt} G(x-y)m(y)dy+\Hh,\ x\in\Tt\\
m>0,\ \int\limits_{\Tt} m(x)dx=1.
\end{cases}
\end{equation}
\begin{remark}
From here on, we do not differentiate between \eqref{eq: main} and \eqref{eq: main current}. Moreover, we refer to \eqref{eq: main current} as the original problem.
\end{remark}
\begin{remark}
Note that $c$, as a solution parameter, in \eqref{eq: main} is replaced by $j$ in \eqref{eq: main current}. We discuss the relation between $c$ and $j$ in Section \ref{sec:current}. 
\end{remark}
Following \cite{GNP}, \cite{GNP2} we call \eqref{eq: main current} the \textit{current formulation} of \eqref{eq: main}. There are two possibilities: $j\neq 0$ and $j=0$. We study the simpler case $j=0$ only in Section \ref{sec:current} and focus on the case $j\neq 0$ afterwards.

Our main observation is that when $G$ is a trigonometric polynomial solutions of \eqref{eq: main current} have a certain structure in terms of unknown Fourier coefficients that satisfy a related equation.

More precisely, for $j\neq 0$ denote by $c_j=(j^2/2)^{\frac{1}{\alpha}}$. Furthermore, for $0<\alpha\leq 2$ denote by $\phi_{\alpha}:(0,+\infty)\to \Rr$ the antiderivative of $x\mapsto \frac{c_j}{x^{\frac{1}{\alpha}}}$; that is,
\begin{equation*}
\phi_{\alpha}(x)=\begin{cases}\frac{c_j \alpha}{\alpha-1}x^{\frac{\alpha-1}{\alpha}},\quad \text{if}\quad \alpha \neq 1,\\
c_j \ln x,\quad \text{if}\quad \alpha = 1.
\end{cases}
\end{equation*}
Next, let $\mathcal{C}$ be the set of all points $(a_0,\cdots,a_n,b_1,\cdots,b_n) \in \Rr^{2n+1}$ such that
\begin{align}\label{eq: C}
a_0+\sum\limits_{k=1}^{n}a_k\cos(2\pi k x)+b_k\sin(2\pi k x)-V(x)>0,\ \text{for all}\ x\in\Tt.
\end{align}
Finally, for $(a_0,\cdots,a_n,b_1,\cdots,b_n) \in \mathcal{C}$ define
\begin{align}\label{eq: upperphi}
&\Phi_{\alpha}(a_0,a_1,\cdots,a_n,b_1,\cdots,b_n)=\\ \notag
&\int\limits_{\Tt} \phi_{\alpha}\left(a_0+\sum\limits_{k=1}^{n}a_k\cos(2\pi k y)+b_k\sin(2\pi k y)-V(y)\right)dy.
\end{align}
Then, we prove the following theorem.
\begin{theorem}\label{thm: free for trig}
Suppose that $G$ is a trigonometric polynomial; that is, 
\begin{equation}\label{eq: G_rep}
G(x)=p_0+\sum\limits_{k=1}^{n}p_k \cos(2\pi k x) + \sum\limits_{k=1}^{n}q_k \sin(2\pi k x)
\end{equation}
for some $n\in\Nn$ and $p_0,p_1,\cdots,p_n,q_1,\cdots,q_n \in \Rr$. Then, if $G$ satisfies \eqref{eq: G_mon} and \eqref{eq: G weight} the system \eqref{eq: main current} has a unique smooth solution.

Moreover, the solution $(m,\Hh)$ of \eqref{eq: main current} is given by formulas
\begin{equation}\label{eq: ansatz m}
m(x)=\frac{c_j}{\left(a_0^*+\sum\limits_{k=1}^{n}a_k^*\cos(2\pi k x)+b_k^*\sin(2\pi k x)-V(x)\right)^{\frac{1}{\alpha}}},
\end{equation}
and
\begin{equation}\label{eq: H=a0-p0}
\Hh=a_0^*-p_0,
\end{equation}
where $(a_0^*,a_1^*,\cdots,a_n^*,b_1^*,\cdots,b_n^*)$ is the unique solution of the system
\begin{align}\label{eq: system main for original}
\begin{cases}
\frac{\partial \Phi_{\alpha}}{\partial a_0}&=1,\\
\frac{\partial \Phi_{\alpha}}{\partial a_k}&=\frac{p_k}{p_k^2+q_k^2}a_k+\frac{q_k}{p_k^2+q_k^2}b_k\\
\frac{\partial \Phi_{\alpha}}{\partial b_k}&=\frac{p_k}{p_k^2+q_k^2}b_k-\frac{q_k}{p_k^2+q_k^2}a_k,\ 1\leq k \leq n,
\end{cases}
\end{align}
where $\Phi_{\alpha}$ is given by \eqref{eq: upperphi}.
\end{theorem}
\begin{remark}
Assumptions \eqref{eq: G_mon}, \eqref{eq: G weight} are natural monotonicity assumptions for the coupling $\int\limits_{\Tt}G(x-y)m(y)dy$, and we discuss them in Section \ref{sec:assum}. When $G$ has the form \eqref{eq: G_rep} these assumptions are equivalent to $p_k \geq 0,\ 0\leq k\leq n$ and $p_0>0$, respectively (see Section \ref{sec:trig}).
\end{remark}
Theorem \ref{thm: free for trig} reduces the a priori-infinite-dimensional problem \eqref{eq: main current} to a finite dimensional problem \eqref{eq: system main for original} when the kernel is a trigonometric polynomial. Also, $\Phi_{\alpha}$ is concave, so \eqref{eq: system main for original} corresponds to finding a root of a monotone mapping which is advantageous from the numerical perspective. This reduction is even more substantial, when the kernel $G$ is a symmetrical trigonometric polynomial; that is, $q_k=0$ for $1\leq k\leq n$. In the latter case, \eqref{eq: system main for original} is equivalent to a concave optimization problem. More precisely, we obtain the following corollary.
\begin{corollary}\label{crl:sym_case}
Suppose that $G$ is a symmetrical trigonometric polynomial; that is, 
\begin{equation}\label{eq: G_rep sym}
G(x)=\sum\limits_{k=0}^{n}p_k \cos(2\pi k x)
\end{equation}
for some $n\in \Nn$ and $p_0,p_1,\cdots,p_n\in \Rr$. Then, if $G$ satisfies \eqref{eq: G_mon} and \eqref{eq: G weight} the system \eqref{eq: main current} has a unique smooth solution.

Moreover, the solution $(m,\Hh)$ of \eqref{eq: main current} is given by formulas \eqref{eq: ansatz m} and \eqref{eq: H=a0-p0} where $(a_0^*,a_1^*,\cdots,a_n^*,b_1^*,\cdots,b_n^*)$ is the unique solution of the optimization problem
\begin{equation}\label{eq: variational}
\max\limits_{(a_0,a_1,\cdots,a_n,b_1,\cdots,b_n)\in \mathcal{C}} \left(\Phi_{\alpha}(a_0,a_1,\cdots,a_n,b_1,\cdots,b_n)-a_0-\sum\limits_{k=1}^{n}\frac{1}{2p_k}\left(a_k^2+b_k^2\right)\right).
\end{equation}
\end{corollary}
Additionally, we find closed form solutions in some special cases.
\begin{theorem}\label{prp: alpha=1}
Assume that $\alpha=1$ and $G,\ V$ are first-order trigonometric polynomials; that is
\begin{align*}
G(x)=p_0+p_1\cos(2\pi x)+q_1\sin(2\pi x),\ V(x)=v_0+v_1\cos(2\pi x)+w_1\sin(2\pi x),
\end{align*}
where $p_0,p_1,q_1,v_0,v_1,w_1 \in \Rr$ and $p_0>0,\ p_1\geq 0$. Then, define $a_0,a_1,b_1,\Hh$ as follows:
\begin{equation}\label{eq: alpha=1}
\begin{cases}
a_0&=2r-1,\\
\Hh&=\frac{j^2(2r-1)}{2}+v_0-p_0\\
a_1&=-\frac{2r(v_1(p_1+j^2r)+w_1q_1)}{(p_1+j^2r)^2+q_1^2},\\
b_1&=-\frac{2r(w_1(p_1+j^2r)-v_1q_1)}{(p_1+j^2r)^2+q_1^2},
\end{cases}
\end{equation}
where $r$ is the unique number that satisfies the following equation 
\begin{equation}\label{eq: r}
\left(1-\frac{1}{r}\right)\left(\left(p_1+j^2r\right)^2+q_1^2\right)=v_1^2+w_1^2\quad \text{and}\quad r\geq 1.
\end{equation}
Then the pair $(m(x),\Hh)$, where
\begin{equation}\label{eq: 44}
m(x)=\frac{1}{a_0+a_1\cos(2\pi x)+b_1\sin(2\pi x)},
\end{equation}
is the unique solution of \eqref{eq: main current}.
\end{theorem}

Besides the trigonometric-polynomial case we also study \eqref{eq: main current} for general $G$. In the latter case, we approximate $G$ by trigonometric polynomials and recover the solution of \eqref{eq: main current} as the limit of solutions of approximate problems. More precisely, we prove the following theorem.
\begin{theorem}\label{thm: main general}
Let $G \in C^2(\Tt),\ V\in C^2(\Tt)$ and $G$ satisfies \eqref{eq: G_mon}, \eqref{eq: G weight}. Then, there exists a sequence of trigonometric polynomials $\{G_n\}_{n\in \Nn}$ such that
\begin{itemize}
\item[\textit{i.}] $G_n$ satisfies \eqref{eq: G_mon} and \eqref{eq: G weight} for all $n\in \Nn$,
\item[\textit{ii.}]$\lim\limits_{n\to \infty}\|G-G_n\|_{C^2(\Tt)}=0$.
\end{itemize}
Furthermore, for $n\in\Nn$ denote by $(m_n,\Hh_n) \in C^2(\Tt)\times \Rr$ the solution of \eqref{eq: main current} corresponding to $G_n$ (the existence of this solution is guaranteed by Theorem \ref{thm: free for trig}). Then, there exists $(m,\Hh) \in C^2(\Tt)\times \Rr$ such that
\begin{equation}\label{eq: 23}
\begin{cases}
\lim\limits_{n\to \infty}\|m-m_n\|_{C^2(\Tt)}=0,\\
\lim\limits_{n\to \infty}(\Hh-\Hh_n)=0.
\end{cases}
\end{equation}
Consequently, $(m,\Hh)$ is the unique smooth solution of \eqref{eq: main current} corresponding to $G$.
\end{theorem}
In combination with preceding results this previous theorem provides a convenient method for numerical calculations of solutions of \eqref{eq: main current}.

We also present a possible way to apply our methods to more general one-dimensional MFG models. We consider the following generalization of \eqref{eq: main}
\begin{equation}\label{eq: main extension}
\begin{cases}
H(x,u_x,m)=\mathcal{F}\left(\int\limits_{\Tt} G(x-y)m(y)dy\right)+\Hh,\\
(mH'_p(x,u_x,m))_x=0,\\
m>0,\ \int\limits_{\Tt} m(x)dx=1.
\end{cases}
\end{equation}
In \eqref{eq: main extension}, $G$ is a given kernel, $H:\Tt\times \Rr\times\Rr^+ \to \Rr,\ (x,p,m)\mapsto H(x,p,m)$, is a given Hamiltonian, and $\mathcal{F}:X \to \Rr$ is a given coupling, where $X$ can be a functional space or $\Rr$. We discuss, formally, how our techniques apply to models like \eqref{eq: main extension}.

The paper is organized as follows. In Section \ref{sec:assum} we present the main assumptions and notation. In Section \ref{sec:current} we study \eqref{eq: main current} for the case $j=0$.

Next, in Section \ref{sec:trig} we analyze \eqref{eq: main current} when $G$ is a trigonometric polynomial and prove Theorem \ref{thm: free for trig}, Corollary \ref{crl:sym_case} and Theorem \ref{prp: alpha=1}. In Section \ref{sec:stability} we analyze \eqref{eq: main current} for a general $G$ and prove Theorem \ref{thm: main general}.

In Section \ref{sec:num} we present some numerical experiments. Finally, in Section \ref{sec:extensions} we discuss possible extensions of our results and a future work.

\section{Assumptions}\label{sec:assum}

Throughout the paper we assume that $G \in C^2(\Tt), V\in C^2(\Tt)$. Moreover, we always assume that
\begin{equation}\label{eq: G_mon}
\int\limits_{\Tt^2} G(x-y) f(x)f(y)dxdy \geq 0,
\end{equation}
for all $f\in C(\Tt)$ and
\begin{equation}\label{eq: G weight}
\int\limits_{\Tt} G(y)dy>0.
\end{equation}
Denote by $G[m](x)=\int\limits_{\Tt}G(x-y)m(y)dy$ the coupling in \eqref{eq: main}. Then, \eqref{eq: G_mon} is equivalent to the condition
\begin{equation*}
	\langle G[m_2]-G[m_1],m_2-m_1\rangle \geq 0,\ \text{for all}\ m_1,m_2,
\end{equation*}
that is the monotonicity of the coupling $G[m]$ and plays an essential role in our analysis. In general, monotonicity of the coupling is fundamental in the regularity theory for MFGs: system \eqref{eq: mfg} degenerates in several directions if the coupling is not monotone. In the view of \eqref{eq: control problem} monotonicity means that agents prefer sparsely populated areas. See \cite{cirant'16} and \cite{GNP} for a systematic study of non-monotone MFGs.

Assumption \eqref{eq: G weight} is a technical assumption. It is not restrictive since one can always modify the kernel by adding a positive constant.

Furthermore, we assume that
\begin{equation*}
	0<\alpha \leq 2.
\end{equation*}
This, also, is a natural assumption for MFGs from the regularity theory perspective. The, now standard, uniqueness proof for MFG systems in \cite{ll3} is valid only for $\alpha$ in this range. This is a strong indication of degeneracy for $\alpha$ outside of this range (which is observed and discussed in detail in \cite{GNP2}). In fact, our methods also reflect these limitations in a natural way.

\section{The 0-current case}\label{sec:current}

As we have pointed out in the Introduction, \eqref{eq: main} can be reduced to \eqref{eq: main current} by eliminating $u$ from the second equation. The analysis of \eqref{eq: main current} is completely different for the case $j=0$ and for the case $j\neq 0$. In fact, the case $j=0$ is much simpler to analyze. Nevertheless, it is more degenerate. In this section, we discuss the case $j=0$.

Firstly, we observe that $j=0$ can occur only when $c=0$. Recall that in this paper we are concerned only with smooth solutions. Therefore, if $(u,m,\Hh)$ is a solution of \eqref{eq: main} we obtain \eqref{eq: reduction} and
\begin{align*}
	c=\int\limits_{\Tt}(u_x(x)+c) dx=j\int\limits_{\Tt}\frac{dx}{m(x)^{\alpha-1}}.
\end{align*}
Hence, $j=0$ if and only if $c=0$. Furthermore, if $c=j=0$ \eqref{eq: main} reduces to
\begin{equation}\label{eq: main current 0}
\begin{cases}
V(x)=\int\limits_{\Tt} G(x-y)m(y)dy+\Hh,\ x\in\Tt\\
m>0,\ \int\limits_{\Tt} m(x)dx=1.
\end{cases}
\end{equation}
At this point, we drop assumptions \eqref{eq: G_mon} and \eqref{eq: G weight} because they are irrelevant. Suppose that $V,G,m$ have following Fourier expansions
\begin{align*}
	V(x)&=\sum\limits_{k=0}^{\infty}v_k \cos(2\pi k x)+\sum\limits_{k=1}^{\infty} w_k \sin(2\pi k x)\\ \notag
	G(x)&=\sum\limits_{k=0}^{\infty}p_k \cos(2\pi k x)+\sum\limits_{k=1}^{\infty} q_k \sin(2\pi k x)\\ \notag
	m(x)&=\sum\limits_{k=0}^{\infty}a_k \cos(2\pi k x)+\sum\limits_{k=1}^{\infty} b_k \sin(2\pi k x)
\end{align*}
Then, \eqref{eq: main current 0} is equivalent to
\begin{equation}\label{eq: 47}
	\begin{cases}
	v_0&=p_0a_0+\Hh\\
	v_k&=\frac{1}{2}(p_k a_k-q_k b_k)\\
	w_k&=\frac{1}{2}(p_k b_k+q_k a_k),\ k\geq 1\\
	a_0&=1\\
	m&>0.
	\end{cases}
\end{equation}
Therefore, we get that
\begin{equation}\label{eq: Hbar current 0}
	\Hh=v_0-p_0,
\end{equation}
and
\begin{equation*}
	\begin{cases}
	a_k&=\frac{2(p_k v_k+q_kw_k)}{p_k^2+q_k^2}\\
	b_k&=\frac{2(p_k w_k-q_kv_k)}{p_k^2+q_k^2},\ k\geq 1.
	\end{cases}
\end{equation*}
Hence, formally, if $p=0$ and $V,G$ are given, we obtain that $\Hh$ is given by \eqref{eq: Hbar current 0} and
\begin{align}\label{eq: m current 0}
	m(x)&=1+\sum\limits_{k=1}^{\infty}\frac{2}{p_k^2+q_k^2}\left((p_k v_k+q_kw_k)\cos(2\pi kx)+(p_k w_k-q_kv_k)\sin(2\pi kx)\right)\\ \notag
	&=1+\sum\limits_{k=1}^{\infty}\frac{2}{r_k}\left(v_k\cos(2\pi kx+\theta_k)+w_k\sin(2\pi kx+\theta_k)\right),
\end{align}
where $r_k e^{i \theta_k}=p_k+i q_k$.

Nevertheless, there are several issues in the previous analysis. Firstly, \eqref{eq: 47} may fail to have solutions or may have infinite number of solutions. If $p_k^2+q_k^2=0$ for some $k\geq 1$ and $v_k^2+w_k^2>0$ then \eqref{eq: 47} and \eqref{eq: main current 0} do not have solutions. On the other hand if $p_k^2+q_k^2=v_k^2+w_k^2=0$ then $a_k,b_k$ can be chosen arbitrarily, so \eqref{eq: 47} and \eqref{eq: main current 0} may have infinite number of solutions. Thus, if $p_k^2+q_k^2=0$ for some $k\geq 1$ \eqref{eq: main current} degenerates in different ways when $j=0$.

Furthermore, if $p_k^2+q_k^2>0$ for all $k \geq 1$ then $m$, at least formally, is given by \eqref{eq: m current 0}. Here we face two potential problems. First, one has to make sense of the formula \eqref{eq: m current 0}. In other words, the series in \eqref{eq: m current 0} may not be summable in any appropriate sense. Moreover, summability of \eqref{eq: m current 0} is a delicate issue and strongly depends on the relation between $\{v_k,w_k\}$ and $\{p_k,q_k\}$.

Additionally, even if the series \eqref{eq: m current 0} converge to a smooth function, we still have the necessary condition $m(x)>0$ and that might fail depending on $V$ and $G$. For instance, if $V$ is such that $v_k=w_k=0$ for all $k\geq 2$, and $G$ is such that $p_k^2+q_k^2>0$ for all $k \geq 1$ we get that
\begin{equation*}
	m(x)=1+\frac{2}{r_1}\left(v_1\cos(2\pi x+\theta_1)+w_1\sin(2\pi x+\theta_1)\right).
\end{equation*}
Therefore,
\begin{equation*}
	\min\limits_{x\in \Tt}m(x)=1-\frac{2\sqrt{v_1^2+w_1^2}}{r_1}=1-\frac{2\sqrt{v_1^2+w_1^2}}{\sqrt{p_1^2+q_1^2}}>0
\end{equation*}
if and only if $p_1^2+q_1^2>4(v_1^2+w_1^2)$. Hence, if the latter is violated \eqref{eq: main current 0} does not have smooth solutions.

Thus, existence of smooth, positive solutions for \eqref{eq: main current 0} depends on peculiar properties of $V$ and $G$. This is quite different in the case $j\neq 0$, where \eqref{eq: main current} obtains smooth, positive solutions under general assumptions on $V$ and $G$.

\section{$G$ is a trigonometric polynomial}\label{sec:trig}

From here on, we assume that $j\neq0$. In this section, our main goal is to prove Theorem \ref{thm: free for trig}, Corollary \ref{crl:sym_case} and Theorem \ref{prp: alpha=1}.

We break the proof of Theorem \ref{thm: free for trig} into three steps. Firstly, we show that \eqref{eq: main current} is equivalent to \eqref{eq: system main for original} - Proposition \ref{thm:equiv}. Secondly, we prove that \eqref{eq: system main for original} has at most one solution - Proposition \ref{prp: system reduced uniqueness}. And thirdly, we show that \eqref{eq: system main for original} has at least one solution - Proposition \ref{prp: existence reduced system}.

We use a short-hand notation $(\bx,\by)$ for a vector $(x_0,x_1,\cdots,x_n,y_1,y_2,\cdots,y_n) \in \Rr^{2n+1}$, where $\bx=(x_0,x_1,\cdots,x_n)$ and $\by=(y_1,y_2,\cdots,y_n)$. For every $\bx=(x_0,x_1,\cdots,x_n)\in \Rr^{n+1}$ we denote by $\bx'=(x_1,\cdots,x_n) \in \Rr^n$.

Here we perform the analysis in terms of Fourier coefficients of $G$. Hence, we formulate assumptions \eqref{eq: G_mon}, \eqref{eq: G weight} in terms of these coefficients.
\begin{lemma}
For a $G$ given by \eqref{eq: G_rep} the assumption \eqref{eq: G_mon} is equivalent to
\begin{equation}\label{eq: >=0 coeffs}
p_k \geq 0,\quad \text{for}\quad 0\leq k \leq n.
\end{equation}
Furthermore, the assumption $\eqref{eq: G weight}$ is equivalent to
\begin{equation}\label{eq: p_0>0}
p_0>0.
\end{equation}
\end{lemma}
\begin{proof}
Let $f\in C(\Tt)$ and
\begin{align*}
c_0&=\int\limits_{\Tt} f(x)dx,\\
c_k&=2\int\limits_{\Tt}f(x) \cos(2\pi kx)dx,\quad s_k=2\int\limits_{\Tt}f(x) \sin(2\pi kx)dx,\quad k\geq 1.
\end{align*}
A straightforward computation yields 
\begin{align*}
\int\limits_{\Tt^2} G(x-y) f(x)f(y)dxdy=p_0c_0^2+\frac{1}{4}\sum\limits_{k=1}^{n}p_k(c_k^2+s_k^2),\ \int\limits_{\Tt} G(x)dx=p_0.
\end{align*}
The rest of the proof is evident.
\end{proof}
\begin{remark}
From here on, we assume that \eqref{eq: >=0 coeffs} and \eqref{eq: p_0>0} hold.
\end{remark}

\begin{proposition}[Equivalent formulation]\label{thm:equiv}
Let $(m,\Hh)\in C(\Tt)\times \Rr$ be a solution of \eqref{eq: main current}. Then, $(m,\Hh)$ is given by formulas \eqref{eq: ansatz m} and \eqref{eq: H=a0-p0} for some $(a_0,\cdots,a_n,b_1,\cdots,b_n) \in \mathcal{C}$ that is a solution of \eqref{eq: system main for original}. 

Conversely, if $(a_0,\cdots,a_n,b_1,\cdots,b_n) \in \mathcal{C}$ is a solution of the system \eqref{eq: system main for original}, then $(m,\Hh)$ defined by \eqref{eq: ansatz m} and \eqref{eq: H=a0-p0} is a solution for \eqref{eq: main current}.
\end{proposition}
\begin{remark}
In our analysis we assume that $p_k^2+q_k^2>0$ for all $1\leq k\leq n$. This assumption is not restrictive and the results are valid even if $p_k=q_k=0$ for some $k\geq 1$. Indeed, if $p_k=q_k=0$, then in \eqref{eq: ansatz} there will be no terms with $\cos(2\pi kx)$ and $\sin(2\pi kx)$ and in the subsequent analysis we just have to omit the trigonometric monomials $\cos(2\pi kx)$ and $\sin(2\pi kx)$.
\end{remark}
\begin{proof}[Proof of Proposition \ref{thm:equiv}]
First, we prove the direct implication. Suppose $(m,\Hh) \in C(\Tt)\times \Rr$ is a solution of \eqref{eq: main current}. A straightforward calculation yields
\begin{align*}
\int\limits_{\Tt} G(x-y)m(y)dy&=p_0u_0+\frac{1}{2}\sum\limits_{k=1}^{n}\left(p_k u_k-q_kv_k\right)\cos(2\pi kx)\\
&+\frac{1}{2}\sum\limits_{k=1}^{n}\left(p_k v_k+q_ku_k\right)\sin(2\pi kx),
\end{align*}
where
\begin{align}\label{eq:m_coeffs}
u_0&=\int\limits_{\Tt} m(x)dx,\\ \nonumber
u_k&=2\int\limits_{\Tt}m(x) \cos(2\pi kx)dx,\quad v_k=2\int\limits_{\Tt}m(x) \sin(2\pi kx)dx,\quad k\geq 1.
\end{align}
Therefore, from \eqref{eq: main current} we obtain
\begin{equation}\label{eq: ansatz}
\frac{j^2}{2m^{\alpha}(x)}+V(x)=a_0^*+\sum\limits_{k=1}^{n}a_k^*\cos(2\pi k x)+b_k^*\sin(2\pi k x),
\end{equation}
which is equivalent to \eqref{eq: ansatz m}. The coefficients $\{a_k,b_k\}$ in the previous equation are given by the formulas
\begin{align}\label{eq:conv_coeffs}
a_0^*&=p_0u_0+\Hh,\\ \nonumber
a_k^*&=\frac{1}{2}(p_ku_k-q_kv_k),\quad b_k^*=\frac{1}{2}(p_kv_k+q_ku_k),\quad 1\leq k\leq n.
\end{align}
Since $m>0$ we obtain that $(a_0^*,\cdots,a_n^*,b_1^*,\cdots,b_n^*) \in \mathcal{C}$.

Furthermore, from \eqref{eq:m_coeffs} and \eqref{eq:conv_coeffs} we obtain that $u_0=1$ and \eqref{eq: H=a0-p0}. Next, we plug the expression \eqref{eq: ansatz m} for $m$ in \eqref{eq:m_coeffs}, and from \eqref{eq:conv_coeffs} we obtain the following system
\begin{align}\label{eq: system for original}
\begin{cases}
1&=\int\limits_{\Tt}\frac{c_jdy}{\left(a_0^*+\sum\limits_{k=1}^{n}a_k^*\cos(2\pi k y)+b_k^*\sin(2\pi k y)-V(y)\right)^{\frac{1}{\alpha}}},\\
a_k^*&=p_k\int\limits_{\Tt}\frac{c_j\cos(2\pi k y)dy}{\left(a_0^*+\sum\limits_{k=1}^{n}a_k^*\cos(2\pi k y)+b_k^*\sin(2\pi k y)-V(y)\right)^{\frac{1}{\alpha}}}\\
&-q_k\int\limits_{\Tt}\frac{c_j\sin(2\pi k y)dy}{\left(a_0^*+\sum\limits_{k=1}^{n}a_k^*\cos(2\pi k y)+b_k^*\sin(2\pi k y)-V(y)\right)^{\frac{1}{\alpha}}},\ 1\leq k \leq n,\\
b_k^*&=p_k\int\limits_{\Tt}\frac{c_j\sin(2\pi k y)dy}{\left(a_0^*+\sum\limits_{k=1}^{n}a_k^*\cos(2\pi k y)+b_k^*\sin(2\pi k y)-V(y)\right)^{\frac{1}{\alpha}}}\\
&+q_k\int\limits_{\Tt}\frac{c_j\cos(2\pi k y)dy}{\left(a_0^*+\sum\limits_{k=1}^{n}a_k^*\cos(2\pi k y)+b_k^*\sin(2\pi k y)-V(y)\right)^{\frac{1}{\alpha}}},\ 1\leq k \leq n.
\end{cases}
\end{align}
Furthermore, note that
\begin{align}\label{eq: Phi derivatives}
\frac{\partial \Phi_{\alpha}}{\partial a_0}&=\int\limits_{\Tt}\frac{c_jdy}{\left(a_0+\sum\limits_{k=1}^{n}a_k\cos(2\pi k y)+b_k\sin(2\pi k y)-V(y)\right)^{\frac{1}{\alpha}}}\\ \notag
\frac{\partial \Phi_{\alpha}}{\partial a_k}&=\int\limits_{\Tt}\frac{c_j\cos(2\pi k y)dy}{\left(a_0+\sum\limits_{k=1}^{n}a_k\cos(2\pi k y)+b_k\sin(2\pi k y)-V(y)\right)^{\frac{1}{\alpha}}}\\ \notag
\frac{\partial \Phi_{\alpha}}{\partial b_k}&=\int\limits_{\Tt}\frac{c_j\sin(2\pi k y)dy}{\left(a_0+\sum\limits_{k=1}^{n}a_k\cos(2\pi k y)+b_k\sin(2\pi k y)-V(y)\right)^{\frac{1}{\alpha}}},
\end{align}
for $1\leq k \leq n$. Therefore, \eqref{eq: system for original} can be written as
\begin{align*}
\begin{cases}
1&= \frac{\partial \Phi_{\alpha}(\ba^*,\bb^*)}{\partial a_0},\\
a_k&=p_k\frac{\partial \Phi_{\alpha}(\ba^*,\bb^*)}{\partial a_k}-q_k\frac{\partial \Phi_{\alpha}(\ba^*,\bb^*)}{\partial b_k}\\
b_k&=p_k\frac{\partial \Phi_{\alpha}(\ba^*,\bb^*)}{\partial b_k}+q_k\frac{\partial \Phi_{\alpha}(\ba^*,\bb^*)}{\partial a_k},\ 1\leq k \leq n,
\end{cases}
\end{align*}
where $(\ba^*,\bb^*)=(a_0^*,\cdots,a_n^*,b_1^*,\cdots,b_n^*)$. But this previous system is equivalent to \eqref{eq: system main for original}.

The proof of the converse implication is the repetition of previous arguments in the reversed order.
\end{proof}

Next, we study some properties of $\mathcal{C}$ and $\Phi_{\alpha}$.

\begin{lemma}\label{lma:upperphi}
The following statements hold.
\begin{itemize}
\item[\textit{i.}] $\mathcal{C}$ is convex and open.
\item[\textit{ii.}] $\Phi_{\alpha} \in C^{\infty}(\mathcal{C})$.
\item[\textit{iii.}] For all $(\ba,\bb) \in \mathcal{C}$
\begin{align}\label{eq:Phi 2nd derivatives}
\frac{\partial \Phi_{\alpha}(\ba,\bb)}{\partial a_l\partial a_r}&=-\frac{c_j}{\alpha}\int\limits_{\Tt}\frac{\cos(2\pi ly)\cos(2\pi ry)dy}{\left(a_0+\sum\limits_{k=1}^{n}a_k\cos(2\pi k y)+b_k\sin(2\pi ky)-V(y)\right)^{1+\frac{1}{\alpha}}}\\ \nonumber
\frac{\partial \Phi_{\alpha}(\ba,\bb)}{\partial b_l\partial b_r}&=-\frac{c_j}{\alpha}\int\limits_{\Tt}\frac{\sin(2\pi ly) \sin (2\pi ry)dy}{\left(a_0+\sum\limits_{k=1}^{n}a_k\cos(2\pi k y)+b_k\sin(2\pi ky)-V(y)\right)^{1+\frac{1}{\alpha}}}\\ \nonumber
\frac{\partial \Phi_{\alpha}(\ba,\bb)}{\partial a_l\partial b_r}&=-\frac{c_j}{\alpha}\int\limits_{\Tt}\frac{\cos(2\pi ly) \sin (2\pi ry)dy}{\left(a_0+\sum\limits_{k=1}^{n}a_k\cos(2\pi k y)+b_k\sin(2\pi ky)-V(y)\right)^{1+\frac{1}{\alpha}}}.
\end{align}
\item[\textit{iv.}] $\Phi_\alpha$ is strictly concave. Moreover, for all $(\ba,\bb) \in \mathcal{C}$ and $(\bxi,\bita) \in \Rr^{2n+1}$ we have that
\begin{align}\label{eq:Hessianform}
&(\bxi,\bita)^T D^2_{\ba,\bb} \Phi_\alpha(\ba,\bb) (\bxi,\bita)\\ \nonumber
=&-\frac{c_j}{\alpha}\int\limits_{\Tt}\frac{\left(\eta_0+\sum\limits_{k=1}^{n}\xi_k\cos(2\pi k y)+\eta_k\sin(2\pi k y)\right)^2dy}{\left(a_0+\sum\limits_{k=1}^{n}a_k\cos(2\pi k y)+b_k\sin(2\pi k y)-V(y)\right)^{1+\frac{1}{\alpha}}}\\ \nonumber
\leq& 0,
\end{align}
with equality if and only if $(\bxi,\bita)=\textbf{0}$.
\item[\textit{v.}] $-\nabla \Phi_\alpha$ is strictly monotone; that is, for all $(\bc_1,\bd_1), (\bc_2,\bd_2) \in \mathcal{C}$
\begin{align}\label{eq:gradPhimon}
\langle\nabla \Phi_\alpha(\boldsymbol{c}_2,\boldsymbol{d}_2)-\nabla \Phi_\alpha(\bc_1,\bd_1), (\bc_2-\bc_1,\bd_2-\bd_1)\rangle \leq 0,
\end{align}
with equality if and only if $(\bc_1,\bd_1)=(\bc_2,\bd_2)$.
\end{itemize}
\end{lemma}
\begin{proof}\textit{i.} This statement is evident.\\
\textit{ii.} This statement is evident.\\
\textit{iii.} We obtain \eqref{eq:Phi 2nd derivatives} by a straightforward calculation.\\
\textit{iv.} Equation \eqref{eq:Hessianform} follows from \eqref{eq:Phi 2nd derivatives} by an algebraic manipulation. Moreover, the equality in \eqref{eq:Hessianform} holds if and only if
\[\eta_0+\sum\limits_{k=1}^{n}\xi_k\cos(2\pi k y)+\eta_k\sin(2\pi k y)=0,\ \text{for all}\ y\in \Tt,
\]
which implies $(\bxi,\bita)=\textbf{0}$.\\
\textit{v.} For $t\in [0,1]$ denote by $(\bc(t),\bd(t))=(1-t)(\bc_1,\bd_1)+t(\bc_2,\bd_2)$. Since $\mathcal{C}$ is convex, we have that $(\bc(t),\bd(t))\in \mathcal{C},t\in[0,1]$. Furthermore, denote by $f(t)=\Phi_\alpha(\bc(t),\bd(t))$. We have that $f\in C^{\infty}([0,1])$ because $\Phi_\alpha \in C^{\infty}(\mathcal{C})$. Moreover, by \eqref{eq:Hessianform} we have that
\begin{align*}
f''(t)=(\bc_2-\bc_1,\bd_2-\bd_1)^T D^2_{\ba,\bb} \Phi_\alpha(\bc(t),\bd(t)) (\bc_2-\bc_1,\bd_2-\bd_1)<0
\end{align*}
for all $t\in[0,1]$, unless $(\bc_1,\bd_1)=(\bc_2,\bd_2)$. Hence, 
\[f'(1)-f'(0) \leq 0,
\]
with equality if and only if $(\bc_1,\bd_1)=(\bc_2,\bd_2)$. We complete the proof by noting that
\begin{align*}
\langle\nabla \Phi_\alpha(\boldsymbol{c}_2,\boldsymbol{d}_2)-\nabla \Phi_\alpha(\bc_1,\bd_1), (\bc_2-\bc_1,\bd_2-\bd_1)\rangle=f'(1)-f'(0).
\end{align*}
\end{proof}

\begin{proposition}[Uniqueness.]\label{prp: system reduced uniqueness}
If $(\bc_1,\bd_1),\ (\bc_2.\bd_2)\in \mathcal{C}$ are solutions of \eqref{eq: system main for original}, then $(\bc_1,\bd_1)=(\bc_2.\bd_2)$.
\end{proposition}
\begin{proof} Let $(\bc_i,\bd_i)=(c_{i0},c_{i1},\cdots,c_{in},d_{i1},\cdots,d_{in})$ for $i=1,2$. Then, we have that
\begin{equation*}
\begin{cases}
\frac{\partial \Phi_{\alpha}(\bc_i,\bd_i)}{\partial a_0}&=1,\\
\frac{\partial \Phi_{\alpha}(\bc_i,\bd_i)}{\partial a_k}&=\frac{p_k}{p_k^2+q_k^2}c_{ik}+\frac{q_k}{p_k^2+q_k^2}d_{ik}\\
\frac{\partial \Phi_{\alpha}(\bc_i,\bd_i)}{\partial b_k}&=\frac{p_k}{p_k^2+q_k^2}d_{ik}-\frac{q_k}{p_k^2+q_k^2}c_{ik},\ 1\leq k \leq n,
\end{cases}
\end{equation*}
for $i=1,2$. Hence,
\begin{align*}
S:=&\langle\nabla \Phi_\alpha(\boldsymbol{c}_2,\boldsymbol{d}_2)-\nabla \Phi_\alpha(\bc_1,\bd_1), (\bc_2-\bc_1,\bd_2-\bd_1)\rangle\\
=&\sum\limits_{k=1}^{n} \frac{p_k}{p_k^2+q_k^2}(c_{2k}-c_{1k})^2+\frac{p_k}{p_k^2+q_k^2}(d_{2k}-d_{1k})^2 \geq 0.
\end{align*}
On the other hand, from \eqref{eq:gradPhimon} we have that $S\leq 0$, so $S=0$ and $(\bc_1,\bd_1)=(\bc_2,\bd_2)$.
\end{proof}

\begin{lemma}\label{lma:omega}
For every $(\ba',\bb)\in \Rr^{2n}$ there exists a unique $a_0=\omega(\ba',\bb) \in \Rr$ such that
\begin{equation}\label{eq:omega}
\frac{\partial\Phi_{\alpha}(\omega(\ba',\bb),\ba',\bb)}{\partial a_0}=1.
\end{equation}
Furthermore, $\omega \in C^{\infty}(\Rr^{2n})$.
\begin{proof}
Fix a point $(\ba',\bb)\in\Rr^{2n}$. Denote by
\begin{align*}l(\ba',\bb)&=-\inf \limits_{x\in \Tt} \left(\sum\limits_{k=1}^{n}\left(a_k\cos(2\pi k x)+b_k\sin(2\pi k x)\right)-V(x)\right)\\
&= -\sum\limits_{k=1}^{n}\left(a_k\cos(2\pi k x_0)+b_k\sin(2\pi k x_0)\right)+V(x_0),
\end{align*}
where $x_0\in \Tt$.
Firstly, we show that
\[\lim\limits_{a_0 \to l(\ba',\bb)} \frac{\partial \Phi_{\alpha}(a_0,\ba',\bb)}{\partial a_0}=\infty.
\]
Denote by
\[f(x)=a_0+\sum\limits_{k=1}^{n}\left(a_k\cos(2\pi k x)+b_k\sin(2\pi k x)\right)-V(x).
\]
Then we have that $x_0 \in \argmin f$. Hence, $f'(x_0)$=0, and
\begin{align*}
f(x)\leq f(x_0)+C(x-x_0)^2=a_0-l(\ba',\bb)+C(x-x_0)^2,
\end{align*}
where $C=C(\ba',\bb)=\frac{1}{2}\sup\limits_{x\in \Tt} |f''(x)|$. Therefore, we have that
\begin{align*}
\frac{\partial \Phi_{\alpha}(a_0,\ba',\bb)}{\partial a_0}&=c_j\int\limits_{\Tt}\frac{dx}{f(x)^{1/\alpha}}=c_j\int\limits_{x_0-1/2}^{x_0+1/2}\frac{dx}{f(x)^{1/\alpha}}\\
&\geq c_j\int\limits_{x_0-1/2}^{x_0+1/2}\frac{dx}{\left(a_0-l(\ba',\bb)+C(x-x_0)^2\right)^{1/\alpha}}\\
&=c_j\int\limits_{-1/2}^{1/2}\frac{dx}{\left(a_0-l(\ba',\bb)+Cx^2\right)^{1/\alpha}}\\
&=c_j\left(a_0-l(\ba',\bb)\right)^{\frac{1}{2}-\frac{1}{\alpha}}\int\limits_{-\frac{1}{2\sqrt{(a_0-l(\ba',\bb))}}}^{\frac{1}{2\sqrt{(a_0-l(\ba',\bb))}}}\frac{dx}{\left(1+Cx^2\right)^{1/\alpha}}.
\end{align*}
For $0<\alpha \leq 2$ we have that
\begin{equation*}
\lim\limits_{\delta \to 0} \delta^{\frac{1}{2}-\frac{1}{\alpha}}\int\limits_{-\frac{1}{2\sqrt{\delta}}}^{\frac{1}{2\sqrt{\delta}}}\frac{dx}{\left(1+Cx^2\right)^{1/\alpha}}=\infty,
\end{equation*}
so
\[\lim\limits_{a_0 \to l(\ba',\bb)} \frac{\partial \Phi_{\alpha}(a_0,\ba',\bb)}{\partial a_0}=\infty.
\]
Finally, the mapping $a_0 \mapsto \frac{\partial \Phi_{\alpha}(a_0,\ba',\bb)}{\partial a_0}$ is decreasing and
\[\lim\limits_{a_0 \to \infty} \frac{\partial \Phi_{\alpha}(a_0,\ba',\bb)}{\partial a_0}=0,
\]
so there exists unique $a_0=\omega(\ba',\bb)$ such that \eqref{eq:omega} holds.	Regularity of $\omega$ follows from the implicit function theorem.
\end{proof}
\end{lemma}
\begin{proposition}[Existence.]\label{prp: existence reduced system}
Let $F:\Rr^{2n} \to \Rr$ be the following function:
\begin{align}\label{eq:F}
F(\ba',\bb)&=\frac{1}{2}\sum\limits_{k=1}^{n}\left(\frac{\partial \Phi_{\alpha}(\omega(\ba',\bb),\ba',\bb)}{\partial a_k}-\frac{p_k}{p_k^2+q_k^2}a_{k}-\frac{q_k}{p_k^2+q_k^2}b_{k}\right)^2\\ \nonumber
&+\frac{1}{2}\sum\limits_{k=1}^{n}\left(\frac{\partial \Phi_{\alpha}(\omega(\ba',\bb),\ba',\bb)}{\partial b_k}-\frac{p_k}{p_k^2+q_k^2}b_{k}+\frac{q_k}{p_k^2+q_k^2}a_{k}\right)^2,
\end{align}
where $(\ba',\bb)=(a_1,a_2,\cdots,a_n,b_1,b_2,\cdots,b_n)$ and $\omega$ is the function from Lemma \ref{lma:omega}. Then, $F$ is bounded by below and coercive. Consequently, the minimization problem
\begin{equation}\label{eq:minF}
\min\limits_{(\ba',\bb)\in \Rr^{2n}} F(\ba',\bb)
\end{equation}
admits at least one solution.

Moreover, if $(\ba',\bb)$ is a critical point for $F$, then $(\ba,\bb)=(\omega(\ba',\bb),\ba',\bb)$ is a solution of \eqref{eq: system main for original}. Therefore, \eqref{eq: system main for original} admits at least one solution.
\end{proposition}
\begin{proof}
Firstly, we show that $F$ from \eqref{eq:F} is coercive and bounded by below. Evidently, $F\geq 0$. Next, from \eqref{eq: Phi derivatives} we have that for all $(\ba',\bb) \in \Rr^{2n}$
\begin{equation*}
\left|\frac{\partial \Phi_{\alpha}(\omega(\ba',\bb),\ba',\bb)}{\partial a_k}\right|,\ \left|\frac{\partial \Phi_{\alpha}(\omega(\ba',\bb),\ba',\bb)}{\partial b_k}\right|\leq\frac{\partial \Phi_{\alpha}(\omega(\ba',\bb),\ba',\bb)}{\partial a_0}=1
\end{equation*}
for $1\leq k \leq n$. Furthermore, we use the elementary inequality
\[(x-y)^2\geq \frac{x^2}{2}-y^2\geq \frac{x^2}{2}-1,\ \text{for}\ x \in \Rr,\ |y|\leq 1,\] and obtain
\begin{align*}
F(\ba',\bb)&\geq\frac{1}{2}\sum\limits_{k=1}^{n}\left[\frac{1}{2}\left(\frac{p_k}{p_k^2+q_k^2}a_{k}+\frac{q_k}{p_k^2+q_k^2}b_{k}\right)^2-1\right]\\ \nonumber
&+\frac{1}{2}\sum\limits_{k=1}^{n}\left[\frac{1}{2}\left(\frac{p_k}{p_k^2+q_k^2}b_{k}-\frac{q_k}{p_k^2+q_k^2}a_{k}\right)^2-1\right]\\
&=\frac{1}{4}\sum\limits_{k=1}^{n}\frac{a_k^2+b_k^2}{p_k^2+q_k^2}-n,
\end{align*}
for all $(\ba',\bb)\in \Rr^{2n}$. Therefore, $F$ is coercive.

Now, we prove that for every critical point $(\ba',\bb)$ of $F$ the point $(\omega(\ba',\bb),\ba',\bb)$ is a solution of \eqref{eq: system main for original}. For $1\leq k\leq n$ denote by
\begin{align*}
\xi_k&=\frac{\partial \Phi_{\alpha}(\omega(\ba',\bb),\ba',\bb)}{\partial a_k}-\frac{p_k}{p_k^2+q_k^2}a_{k}-\frac{q_k}{p_k^2+q_k^2}b_{k}\\ \nonumber
\eta_k&=\frac{\partial \Phi_{\alpha}(\omega(\ba',\bb),\ba',\bb)}{\partial b_k}-\frac{p_k}{p_k^2+q_k^2}b_{k}+\frac{q_k}{p_k^2+q_k^2}a_{k}.
\end{align*}
Then,
\begin{align*}
\frac{\partial F(\ba',\bb)}{\partial a_l}&=\sum\limits_{k=1}^{n}\xi_k\left(\frac{\partial^2 \Phi_{\alpha}\left(\omega(\ba',\bb),\ba',\bb\right)}{\partial a_k\partial a_l}-\frac{p_k\delta_{kl}}{p_k^2+q_k^2}+\frac{\partial^2 \Phi_{\alpha}\left(\omega(\ba',\bb),\ba',\bb\right)}{\partial a_k\partial a_0}\cdot\frac{\partial \omega}{\partial a_l}\right)\\
&+\sum\limits_{k=1}^{n}\eta_k\left(\frac{\partial^2 \Phi_{\alpha}\left(\omega(\ba',\bb),\ba',\bb\right)}{\partial b_k\partial a_l}+\frac{q_k\delta_{kl}}{p_k^2+q_k^2}+\frac{\partial^2 \Phi_{\alpha}\left(\omega(\ba',\bb),\ba',\bb\right)}{\partial b_k\partial a_0}\cdot\frac{\partial \omega}{\partial a_l}\right),\\
\frac{\partial F(\ba',\bb)}{\partial b_l}&=\sum\limits_{k=1}^{n}\xi_k\left(\frac{\partial^2 \Phi_{\alpha}\left(\omega(\ba',\bb),\ba',\bb\right)}{\partial a_k\partial b_l}-\frac{q_k\delta_{kl}}{p_k^2+q_k^2}+\frac{\partial^2 \Phi_{\alpha}\left(\omega(\ba',\bb),\ba',\bb\right)}{\partial a_k\partial a_0}\cdot\frac{\partial \omega}{\partial b_l}\right)\\
&+\sum\limits_{k=1}^{n}\eta_k\left(\frac{\partial^2 \Phi_{\alpha}\left(\omega(\ba',\bb),\ba',\bb\right)}{\partial b_k\partial b_l}-\frac{p_k\delta_{kl}}{p_k^2+q_k^2}+\frac{\partial^2 \Phi_{\alpha}\left(\omega(\ba',\bb),\ba',\bb\right)}{\partial b_k\partial a_0}\cdot\frac{\partial \omega}{\partial b_l}\right),
\end{align*}
for $1\leq l \leq n$. Next, by differentiating \eqref{eq:omega} we obtain
\begin{align*}
\frac{\partial \omega(\ba',\bb)}{\partial a_l}=-\frac{\frac{\partial^2 \Phi_{\alpha}(\omega(\ba',\bb),\ba',\bb)}{\partial a_0\partial a_l}}{\frac{\partial^2 \Phi_{\alpha}(\omega(\ba',\bb),\ba',\bb)}{\partial a_0^2}},\quad \frac{\partial \omega(\ba',\bb)}{\partial b_l}=-\frac{\frac{\partial^2 \Phi_{\alpha}(\omega(\ba',\bb),\ba',\bb)}{\partial a_0\partial b_l}}{\frac{\partial^2 \Phi_{\alpha}(\omega(\ba',\bb),\ba',\bb)}{\partial a_0^2}},
\end{align*}
for $1\leq l \leq n$.

Now, suppose $(\ba',\bb) \in \Rr^n$ is a minimizer of \eqref{eq:minF}.
Then, we have that
\begin{align}\label{eq:criticalpoint}
0&=\sum\limits_{l=1}^{n}\xi_l \frac{\partial F(\ba',\bb)}{\partial a_l}+\sum\limits_{l=1}^{n}\eta_l \frac{\partial F(\ba',\bb)}{\partial b_l}\\ \nonumber
&=\sum\limits_{l,k}\xi_l\xi_k\left(\frac{\partial^2 \Phi_{\alpha}\left(\omega(\ba',\bb),\ba',\bb\right)}{\partial a_k\partial a_l}-\frac{p_k\delta_{kl}}{p_k^2+q_k^2}+\frac{\partial^2 \Phi_{\alpha}\left(\omega(\ba',\bb),\ba',\bb\right)}{\partial a_k\partial a_0}\cdot\frac{\partial \omega}{\partial a_l}\right)\\ \nonumber
&+\sum\limits_{l,k}\xi_l\eta_k\left(\frac{\partial^2 \Phi_{\alpha}\left(\omega(\ba',\bb),\ba',\bb\right)}{\partial b_k\partial a_l}+\frac{q_k\delta_{kl}}{p_k^2+q_k^2}+\frac{\partial^2 \Phi_{\alpha}\left(\omega(\ba',\bb),\ba',\bb\right)}{\partial b_k\partial a_0}\cdot\frac{\partial \omega}{\partial a_l}\right)\\ \nonumber
&+\sum\limits_{l,k}\eta_l \xi_k \left(\frac{\partial^2 \Phi_{\alpha}\left(\omega(\ba',\bb),\ba',\bb\right)}{\partial a_k\partial b_l}-\frac{q_k\delta_{kl}}{p_k^2+q_k^2}+\frac{\partial^2 \Phi_{\alpha}\left(\omega(\ba',\bb),\ba',\bb\right)}{\partial a_k\partial a_0}\cdot\frac{\partial \omega}{\partial b_l}\right)\\ \nonumber
&+\sum\limits_{l,k}\eta_l\eta_k\left(\frac{\partial^2 \Phi_{\alpha}\left(\omega(\ba',\bb),\ba',\bb\right)}{\partial b_k\partial b_l}-\frac{p_k\delta_{kl}}{p_k^2+q_k^2}+\frac{\partial^2 \Phi_{\alpha}\left(\omega(\ba',\bb),\ba',\bb\right)}{\partial b_k\partial a_0}\cdot\frac{\partial \omega}{\partial b_l}\right)\\ \nonumber
&=(\bxi',\bita)^T D^2_{\ba',\bb} \Phi_\alpha(\omega(\ba',\bb),\ba',\bb) (\bxi',\bita)-\sum\limits_{k=1}^{n}\frac{p_k}{p_k^2+q_k^2}\left(\xi_k^2+\eta_k^2\right)\\ \nonumber
&+\sum\limits_{l,k}\left(\xi_l\xi_k \frac{\partial^2 \Phi_{\alpha}\left(\omega(\ba',\bb),\ba',\bb\right)}{\partial a_k\partial a_0}\cdot\frac{\partial \omega}{\partial a_l}+\xi_l\eta_k\frac{\partial^2 \Phi_{\alpha}\left(\omega(\ba',\bb),\ba',\bb\right)}{\partial b_k\partial a_0}\cdot\frac{\partial \omega}{\partial a_l}\right)\\ \nonumber
&+\sum\limits_{l,k}\left(\eta_l\xi_k \frac{\partial^2 \Phi_{\alpha}\left(\omega(\ba',\bb),\ba',\bb\right)}{\partial a_k\partial a_0}\cdot\frac{\partial \omega}{\partial b_l}+\eta_l\eta_k\frac{\partial^2 \Phi_{\alpha}\left(\omega(\ba',\bb),\ba',\bb\right)}{\partial b_k\partial a_0}\cdot\frac{\partial \omega}{\partial b_l}\right)\\ \nonumber
&=(\bxi',\bita)^T D^2_{\ba',\bb} \Phi_\alpha(\omega(\ba',\bb),\ba',\bb) (\bxi',\bita)-\sum\limits_{k=1}^{n}\frac{p_k}{p_k^2+q_k^2}\left(\xi_k^2+\eta_k^2\right)\\ \nonumber
&+\left(\sum\limits_{l=1}^{n}\xi_l \frac{\partial \omega}{\partial a_l}+\eta_l \frac{\partial \omega}{\partial b_l}\right)\left(\sum\limits_{k=1}^{n}\xi_k \frac{\partial \Phi_{\alpha}(\omega(\ba',\bb),\ba',\bb)}{\partial a_k \partial a_0}+\eta_k \frac{\partial \Phi_{\alpha}(\omega(\ba',\bb),\ba',\bb)}{\partial b_k \partial a_0}\right)\\ \nonumber
&=(\bxi',\bita)^T D^2_{\ba',\bb} \Phi_\alpha(\omega(\ba',\bb),\ba',\bb) (\bxi',\bita)-\sum\limits_{k=1}^{n}\frac{p_k}{p_k^2+q_k^2}\left(\xi_k^2+\eta_k^2\right)\\ \nonumber
&-\frac{\partial^2 \Phi_{\alpha}(\omega(\ba',\bb),\ba',\bb)}{\partial a_0^2}\left(\sum\limits_{l=1}^{n}\xi_l \frac{\frac{\partial^2 \Phi_{\alpha}(\omega(\ba',\bb),\ba',\bb)}{\partial a_0\partial a_l}}{\frac{\partial^2 \Phi_{\alpha}(\omega(\ba',\bb),\ba',\bb)}{\partial a_0^2}}+\eta_l \frac{\frac{\partial^2 \Phi_{\alpha}(\omega(\ba',\bb),\ba',\bb)}{\partial a_0\partial b_l}}{\frac{\partial^2 \Phi_{\alpha}(\omega(\ba',\bb),\ba',\bb)}{\partial a_0^2}}\right)^2\\ \nonumber
&=(\bxi',\bita)^T D^2_{\ba',\bb} \Phi_\alpha(\omega(\ba',\bb),\ba',\bb) (\bxi',\bita)-\sum\limits_{k=1}^{n}\frac{p_k}{p_k^2+q_k^2}\left(\xi_k^2+\eta_k^2\right)\\ \nonumber
&-\frac{\partial^2 \Phi_{\alpha}(\omega(\ba',\bb),\ba',\bb)}{\partial a_0^2} \xi_0^2,
\end{align}
where $\bxi'=(\xi_1,\cdots,\xi_n),\ \bita=(\eta_1,\cdots,\eta_n)$, and
\begin{equation*}
\xi_0=-\sum\limits_{l=1}^{n}\left(\xi_l \frac{\frac{\partial^2 \Phi_{\alpha}(\omega(\ba',\bb),\ba',\bb)}{\partial a_0\partial a_l}}{\frac{\partial^2 \Phi_{\alpha}(\omega(\ba',\bb),\ba',\bb)}{\partial a_0^2}}+\eta_l \frac{\frac{\partial^2 \Phi_{\alpha}(\omega(\ba',\bb),\ba',\bb)}{\partial a_0\partial b_l}}{\frac{\partial^2 \Phi_{\alpha}(\omega(\ba',\bb),\ba',\bb)}{\partial a_0^2}}\right).
\end{equation*}
Furthermore, we have that
\begin{align*}
\frac{\partial^2 \Phi_{\alpha}(\omega(\ba',\bb),\ba',\bb)}{\partial a_0^2}\xi_0^2&+2 \sum\limits_{l=1}^{n} \left(\xi_0 \xi_l \frac{\partial^2 \Phi_{\alpha}(\omega(\ba',\bb),\ba',\bb)}{\partial a_0 \partial a_l} +\xi_0 \eta_l \frac{\partial^2 \Phi_{\alpha}(\omega(\ba',\bb),\ba',\bb)}{\partial a_0 \partial b_l}\right)\\
&=\frac{\partial^2 \Phi_{\alpha}(\omega(\ba',\bb),\ba',\bb)}{\partial a_0^2}\xi_0^2-2\frac{\partial^2 \Phi_{\alpha}(\omega(\ba',\bb),\ba',\bb)}{\partial a_0^2}\xi_0^2\\
&=-\frac{\partial^2 \Phi_{\alpha}(\omega(\ba',\bb),\ba',\bb)}{\partial a_0^2}\xi_0^2.
\end{align*}
Hence, from \eqref{eq:criticalpoint} we obtain that
\begin{align*}
0&=(\bxi',\bita)^T D^2_{\ba',\bb} \Phi_\alpha(\omega(\ba',\bb),\ba',\bb) (\bxi',\bita)-\sum\limits_{k=1}^{n}\frac{p_k}{p_k^2+q_k^2}\left(\xi_k^2+\eta_k^2\right)\\ \nonumber
&-\frac{\partial^2 \Phi_{\alpha}(\omega(\ba',\bb),\ba',\bb)}{\partial a_0^2} \xi_0^2\\
&=(\bxi,\bita)^T D^2_{\ba,\bb} \Phi_\alpha(\omega(\ba',\bb),\ba',\bb) (\bxi,\bita)-\sum\limits_{k=1}^{n}\frac{p_k}{p_k^2+q_k^2}\left(\xi_k^2+\eta_k^2\right)\\
&\leq (\bxi,\bita)^T D^2_{\ba,\bb} \Phi_\alpha(\omega(\ba',\bb),\ba',\bb) (\bxi,\bita),
\end{align*}
where $\bxi=(\xi_0,\bxi')$. On the other hand, from \eqref{eq:Hessianform} we have that
\begin{align*}
(\bxi,\bita)^T D^2_{\ba,\bb} \Phi_\alpha(\omega(\ba',\bb),\ba',\bb) (\bxi,\bita) \leq 0.
\end{align*}
Therefore, there is an equality in the previous equation and from \eqref{eq:Hessianform} we obtain that $(\bxi,\bita)=\textbf{0}$ or that, equivalently, $\xi_k=\eta_k=0$ for $1\leq k\leq n$. The latter precisely means that $(\omega(\ba',\bb),\ba',\bb)$ is a solution of \eqref{eq: system main for original}.
\end{proof}
Now we are in the position to prove Theorem \ref{thm: free for trig}.
\begin{proof}[Proof of the Theorem \ref{thm: free for trig}]
We have that \eqref{eq: main current} is equivalent to \eqref{eq: system main for original} by Proposition \ref{thm:equiv}. Furthermore, \eqref{eq: system main for original} admits a solution $(\ba^*,\bb^*)$ by Theorem \ref{prp: existence reduced system}. Moreover, $(\ba^*,\bb^*)$ is unique by Proposition \ref{prp: system reduced uniqueness}. Hence, \eqref{eq: main current} admits a unique solution given by \eqref{eq: ansatz m} and \eqref{eq: H=a0-p0}.
\end{proof}
Next, we prove Corollary \ref{crl:sym_case}.
\begin{proof}[Proof of the Corollary \ref{crl:sym_case}]
By Theorem \ref{thm: free for trig} we have that \eqref{eq: main current} obtains unique solution, $(m,\Hh)$, given by \eqref{eq: ansatz m} and \eqref{eq: H=a0-p0}, where $(\ba^*,\bb^*)=(a_0^*,a_1^*,\cdots,a_n^*,b_1^*\cdots,b_n^*)$ is the unique solution of \eqref{eq: system main for original}. Since $G$ has the form \eqref{eq: G_rep sym} we have that $q_k=0$ for $1\leq k \leq n$. Therefore, \eqref{eq: system main for original} can be written as
\begin{equation*}
\nabla_{\ba,\bb}\left(\Phi_{\alpha}(\ba,\bb)-a_0-\sum\limits_{k=1}^{n}\frac{1}{2p_k}\left(a_k^2+b_k^2\right)\right)\Bigg\rvert_{(\ba,\bb)=(\ba^*,\bb^*)}=0.
\end{equation*}
Furthermore, by Lemma \ref{lma:upperphi} $\Phi_{\alpha}$ is strictly concave on $\mathcal{C}$ (see \eqref{eq: C} for the definition of $\mathcal{C}$), so the function
\begin{equation*}
(\ba,\bb)\mapsto\Phi_{\alpha}(\ba,\bb)-a_0-\sum\limits_{k=1}^{n}\frac{1}{2p_k}\left(a_k^2+b_k^2\right)
\end{equation*}
is also strictly concave on $\mathcal{C}$. Hence, $(\ba^*,\bb^*)$ is the unique maximum of \eqref{eq: variational}.
\end{proof}
Finally, we prove Theorem \ref{prp: alpha=1}.
\begin{proof}[Proof of the Theorem \ref{prp: alpha=1}]
Firstly, note that if $m$ is a solution of \eqref{eq: main current} with $\alpha=1$, then $m$ must necessarily have the form \eqref{eq: 44}. Consequently, \eqref{eq: 44} leads to
\begin{align}\label{eq: 45}
&\frac{j^2}{2}(a_0+a_1\cos(2\pi x)+b_1\sin(2\pi x))+v_0+v_1\cos(2\pi x)+w_1\sin (2\pi x)\\ \notag
=&p_0\int\limits_{\Tt} \frac{dy}{a_0+a_1\cos(2\pi y)+b_1\sin(2\pi y)}+p_1\int\limits_{\Tt} \frac{\cos(2\pi(x-y))dy}{a_0+a_1\cos(2\pi y)+b_1\sin(2\pi y)}\\ \notag
&+q_1\int\limits_{\Tt} \frac{\sin(2\pi(x-y))dy}{a_0+a_1\cos(2\pi y)+b_1\sin(2\pi y)}.
\end{align}
A direct calculation yields the following identities
\begin{align}\label{eq: 46}
\int\limits_{\Tt} \frac{dy}{a_0+a_1\cos(2\pi y)+b_1\sin(2\pi y)}&=\frac{1}{\sqrt{a_0^2-a_1^2-b_1^2}}\\ \notag
\int\limits_{\Tt} \frac{\cos(2\pi y)dy}{a_0+a_1\cos(2\pi y)+b_1\sin(2\pi y)}&=\frac{a_1(\sqrt{a_0^2-a_1^2-b_1^2}-a_0)}{(a_1^2+b_1^2)\sqrt{a_0^2-a_1^2-b_1^2}}\\ \notag
\int\limits_{\Tt} \frac{\sin(2\pi y)dy}{a_0+a_1\cos(2\pi y)+b_1\sin(2\pi y)}&=\frac{b_1(\sqrt{a_0^2-a_1^2-b_1^2}-a_0)}{(a_1^2+b_1^2)\sqrt{a_0^2-a_1^2-b_1^2}}\\ \notag
\end{align}
Using \eqref{eq: 46} in \eqref{eq: 45} and taking into the account that $\int\limits_{\Tt}m(x)dx=1$, we obtain
\begin{equation*}
\begin{cases}
\frac{1}{\sqrt{a_0^2-a_1^2-b_1^2}}&=1\\
\frac{j^2}{2}a_0+v_0&=\frac{p_0}{\sqrt{a_0^2-a_1^2-b_1^2}}+\Hh\\
\frac{j^2}{2}a_1+v_1&=\frac{(\sqrt{a_0^2-a_1^2-b_1^2}-a_0)}{(a_1^2+b_1^2)\sqrt{a_0^2-a_1^2-b_1^2}}(p_1a_1-q_1b_1)\\
\frac{j^2}{2}b_1+w_1&=\frac{(\sqrt{a_0^2-a_1^2-b_1^2}-a_0)}{(a_1^2+b_1^2)\sqrt{a_0^2-a_1^2-b_1^2}}(p_1b_1+q_1a_1).
\end{cases}
\end{equation*}
which can be equivalently written as
\begin{equation*}
\begin{cases}
\Hh&=\frac{j^2}{2}a_0+v_0-p_0\\
\frac{j^2}{2}a_1+v_1&=\frac{(1-a_0)}{(a_0^2-1)}(p_1a_1-q_1b_1)\\
\frac{j^2}{2}b_1+w_1&=\frac{(1-a_0)}{(a_0^2-1)}(p_1b_1+q_1a_1)\\
a_0^2-a_1^2-b_1^2&=1.
\end{cases}
\end{equation*}
We eliminate $a_1$ and $b_1$ in the second and third equations and find $a_0$ from the fourth equation. It is algebraically more appealing to put $a_0=2r-1$. Then, a straightforward calculation yields \eqref{eq: r}.
	
Since $m>0$ we have that $a_0>0$. Moreover, from the fourth equation in the previous system we have that $a_0^2\geq 1$, so $a_0\geq 1$ or $r\geq 1$. Note that the left-hand-side of \eqref{eq: r} is increasing function in $r$ for $r\geq 1$, and it is equal to 0 at $r=1$ and to $\infty$ at $r=\infty$. Therefore, for arbitrary choices of $v_1,w_1$ there is a unique $r\geq 1$ such that \eqref{eq: r} holds. This is coherent with the fact that \eqref{eq: main current} obtains a unique smooth solution. Moreover, \eqref{eq: r} is a cubic equation. Hence, formulas in \eqref{eq: alpha=1} are explicit.
\end{proof}

\section{$G$ is a general kernel}\label{sec:stability}

In this section we prove Theorem \ref{thm: main general}. We divide the proof into two steps. First, we prove that solutions of \eqref{eq: main current} are stable under $C^2$ perturbation of the kernel. Second, we show that arbitrary $C^2$ kernel can be approximated by suitable trigonometric polynomials.
\begin{proof}[Proof of the Theorem \ref{thm: main general}]
The uniqueness of the solution for \eqref{eq: main current} follows from the uniqueness of the solution of \eqref{eq: main} (See \cite{LCDF}).\\
\textbf{Part 1. Stability.} Suppose that $\{G_n\}_{n\in \Nn} \subset C^2(\Tt)$ are such that
\[\lim\limits_{n\to \infty} \|G_n-G\|_{C^2(\Tt)}=0.\]
Moreover, assume that for each $n\geq 1$ \eqref{eq: main current} has a solution, $(m_n,\Hh_n) \in C^2(\Tt) \times \Rr$, corresponding to the kernel $G_n$. We aim to prove that there exists $(m,\Hh) \in C^2(\Tt)\times \Rr$ such that \eqref{eq: 23} holds and $(m,\Hh)$ is the solution of \eqref{eq: main current} corresponding to the kernel $G$.
\begin{remark}
Note that in this part of the proof we do not assume that $\{G_n\}$ are trigonometric polynomials and that they satisfy \eqref{eq: G_mon}, \eqref{eq: G weight}. We need these assumptions in the second part of the proof to guarantee the existence of solutions $(m_n,\Hh_n)$.
\end{remark}
We are going to show that families
\begin{equation*}
\left\{m_n\right\}_{n\in\Nn},\quad \left\{\frac{1}{m_n^{\alpha}}\right\}_{n\in\Nn}
\end{equation*}
are uniformly bounded and equicontinuous. Denote by
\begin{align*}
m_n(x)=\frac{(j^2/2)^{1/\alpha}}{f_n(x)^{1/\alpha}},
\end{align*}
where
\begin{align*}
f_n(x)=\int\limits_{\Tt} G_n(x-y)m_n(y)dy+\Hh_n-V(x),\quad x \in \Tt.
\end{align*}
We have that
\begin{align*}
|f_n''(x)|&=\left|\int\limits_{\Tt} G_n''(x-y)m_n(y)dy-V''(x)\right|\leq \sup \limits_{\Tt} |G_n''| \int\limits_{\Tt} m_n(y)dy+\sup \limits_{\Tt} |V''|\\
&\leq\sup\limits_n \sup \limits_{\Tt} |G_n''| +\sup \limits_{\Tt} |V''|=:2C,\quad x \in \Tt.
\end{align*}
Next, denote by $\sigma_n:=\min \limits_{\Tt} f_n=f_n(x_n),$
for some $x_n \in \Tt$. Then, we have that $f'(x_n)=0$, and
\begin{align*}
f_n(x)\leq \sigma_n+C(x-x_n)^2,\quad x \in \Tt.
\end{align*}
Therefore,
\begin{align*}
1&= \int\limits_{\Tt} m_n(x)dx=(j^2/2)^{1/\alpha} \int\limits_{\Tt}\frac{dx}{f_n(x)^{1/\alpha}}=(j^2/2)^{1/\alpha}\int\limits_{x_n-1/2}^{x_n+1/2}\frac{dx}{f_{n}(x)^{1/\alpha}}\\
&\geq (j^2/2)^{1/\alpha}\int\limits_{x_n-1/2}^{x_n+1/2}\frac{dx}{\left(\sigma_{n}+C(x-x_{n})^2\right)^{1/\alpha}}=(j^2/2)^{1/\alpha}\int\limits_{-1/2}^{1/2}\frac{dx}{\left(\sigma_{n}+Cx^2\right)^{1/\alpha}}\\
&=(j^2/2)^{1/\alpha}\sigma_n^{\frac{1}{2}-\frac{1}{\alpha}}\int\limits_{-\frac{1}{2\sqrt{\sigma_n}}}^{\frac{1}{2\sqrt{\sigma_n}}}\frac{dx}{\left(1+Cx^2\right)^{1/\alpha}}.
\end{align*}
Furthermore, for $0<\alpha\leq 2$ we have that
\begin{equation*}
\lim\limits_{\delta \to 0} \delta^{\frac{1}{2}-\frac{1}{\alpha}}\int\limits_{-\frac{1}{2\sqrt{\delta}}}^{\frac{1}{2\sqrt{\delta}}}\frac{dx}{\left(1+Cx^2\right)^{1/\alpha}}=\infty.
\end{equation*}
Therefore, $\sigma_n\geq \delta_0>0,$
or
\begin{equation}\label{eq: 37}
m_n(x)=\frac{(j^2/2)^{1/\alpha}}{f_n(x)^{1/\alpha}}\leq \frac{(j^2/2)^{1/\alpha}}{\sigma_{n}^{1/\alpha}}\leq \frac{(j^2/2)^{1/\alpha}}{\delta_0^{1/\alpha}}=:C_1,\quad x \in \Tt,
\end{equation}
for $n \geq 1$. Furthermore, denote by $m_n(z_n)=\max \limits_{\Tt} m_n.$ Then, we have that
\begin{equation}\label{eq: 35}
m_n(z_n)\geq \int\limits_{\Tt} m_n(x)dx=1.
\end{equation}
Furthermore, for every $x,z \in \Tt$ we have that
\begin{align}\label{eq: 34}
\left|\frac{j^2}{2m_n^{\alpha}(x)}-\frac{j^2}{2m_n^{\alpha}(z)}\right|&=\left|\int\limits_{\Tt} (G_n(x-y)-G_n(z-y))m_n(y)dy-V(x)+V(z)\right|\\ \notag
&\leq \left(\sup\limits_{\Tt}|G'_n|+\sup\limits_{\Tt}|V'|\right)|x-z|\\ \notag
&\leq \left(\sup\limits_{n,\Tt}|G'_n|+\sup\limits_{\Tt}|V'|\right)|x-z|.
\end{align}
Firstly, if we plug in $z=z_n$ in \eqref{eq: 34} and use \eqref{eq: 35}, we get that
\begin{equation*}
\frac{1}{m_n^{\alpha}(x)}\leq C_2,\quad x \in \Tt,
\end{equation*}
for all $n\geq 1$. Secondly, \eqref{eq: 34} yields that the family
\[\left\{\frac{1}{m_n^{\alpha}}\right\}_{n\in\Nn}
\]
is uniformly Lipschitz which in turn yields (in combination with \eqref{eq: 37}) that the family $\left\{m_n^{\alpha}\right\}_{n\in\Nn}$ is also uniformly Lipschitz.

Since families \eqref{eq: 23} are uniformly bounded, we get that $\{\Hh_n\}_{n \in \Nn}$ is a bounded sequence. Then, we can assume that there exists $(m,\Hh) \in C(\Tt)\times \Rr$ such that
\begin{align*}
\lim \limits_{n\to \infty} \|m_n-m\|_{C(\Tt)}&=0,\\
\lim \limits_{n\to \infty} \left\|\frac{1}{m_n^\alpha}-\frac{1}{m^\alpha}\right\|_{C(\Tt)}&=0,\\
\lim \limits_{n\to \infty} (\Hh_n-\Hh)&=0,
\end{align*}
through a subsequence. Moreover, we obtain \eqref{eq: 23} through the same subsequence.

From the previous equations, we obtain that $(m,\Hh)$ solves \eqref{eq: main current} for the kernel $G$. Next, \eqref{eq: main current} must have a unique solution because it is equivalent to \eqref{eq: main} that can have at most one solution (see \cite{LCDF}). Hence, the limit, $(m,\Hh)$, is the same for all subsequences. Therefore, \eqref{eq: 23} is valid through the whole sequence.\\
\textbf{Part 2. Approximation.} Suppose $G \in C^2(\Tt)$ satisfies \eqref{eq: G_mon} and \eqref{eq: G weight} are satisfied. We formally expand $G$ in Fourier series 
\begin{align*}
G(x)=p_0+\sum\limits_{k=1}^{\infty} p_k \cos(2\pi k x)+q_k \sin(2\pi k x),\quad x \in \Tt.
\end{align*}
Denote by 
\begin{align*}
S_n(x)=p_0+\sum\limits_{k=1}^{n} p_k \cos(2\pi k x)+q_k \sin(2\pi k x),\quad x \in \Tt,\quad n\geq 1,
\end{align*}
and $S_0(x)=p_0$ the truncated Fourier series. Furthermore, let $G_n$ be the corresponding Ces\`{a}ro mean; that is,
\begin{equation*}
G_n(x)=\frac{1}{n+1}\sum\limits_{k=0}^{n} S_k(x)=p_0^n+\sum\limits_{k=1}^{n} p_k^n \cos(2\pi k x)+q_k^n \sin(2\pi k x).
\end{equation*}
Then by Fej\'{e}r's theorem (see Theorem 1.10 in \cite{javier'01}) we have that
\begin{equation*}
\lim\limits_{n\to \infty} \|G_n-G\|_{C^2(\Tt)}=0.
\end{equation*}
Next, $G$ satisfies \eqref{eq: G_mon}, \eqref{eq: G weight} so $p_0>0$ and $p_k \geq 0$ for $k \geq 1$. Therefore, we have that
\begin{align*}
p_0^n&=p_0>0,\\ \notag
p_k^n&=\frac{n+1-k}{n+1}p_k\geq 0, \quad 1\leq k\leq n,
\end{align*}
so $G_n$ also satisfy \eqref{eq: G weight}, \eqref{eq: G_mon} for all $n\geq 1$.

Now, we can complete the proof of Theorem \ref{thm: main general}. We approximate $G$ using Part 2 and conclude using Part 1.
\end{proof}

\section{Numerical solutions}\label{sec:num}

Here, we numerically solve \eqref{eq: main current} for different types of kernels $G$. We present three cases. First, we consider $G$ that is a non-symmetric trigonometric polynomial. Second, we consider $G$ that is a symmetric trigonometric polynomial. And third, we consider $G$ that is periodic but that is not a trigonometric polynomial.

During the whole discussion in this section we assume that
\begin{equation*}
\begin{cases}
V(x)=2 \sin\left(2\pi (x+\frac{1}{4})\right),\ x\in \Tt,\\
\alpha=1.5,\ j=\sqrt{2}.
\end{cases}
\end{equation*}
This choice of parameters in \eqref{eq: main current} is random and robustness of our calculations does not depend on a particular choice of parameters.
\subsection{The case of a non-symmetric trigonometric polynomial}
By Theorem \ref{thm: free for trig} we have that for a given non-symmetric trigonometric polynomial $G$ the solution $m$ of \eqref{eq: main current} has the form \eqref{eq: ansatz m}, where the vector $(a_0^*,a_1^*,\cdots,a_n^*,b_1^*,\cdots,b_n^*)$ is the unique solution of \eqref{eq: system main for original}. Furthermore, we define
\begin{align*}
M(\ba,\bb)&=\left(\frac{\partial \Phi_{\alpha}(\ba,\bb)}{\partial a_0}-1\right)^2+\frac{1}{2}\sum\limits_{k=1}^{n}\left(\frac{\partial \Phi_{\alpha}(\ba,\bb)}{\partial a_k}-\frac{p_k}{p_k^2+q_k^2}a_{k}-\frac{q_k}{p_k^2+q_k^2}b_{k}\right)^2\\ \nonumber
&+\frac{1}{2}\sum\limits_{k=1}^{n}\left(\frac{\partial \Phi_{\alpha}(\ba,\bb)}{\partial b_k}-\frac{p_k}{p_k^2+q_k^2}b_{k}+\frac{q_k}{p_k^2+q_k^2}a_{k}\right)^2,
\end{align*}
where $(\ba,\bb)=(a_0,a_1,\cdots,a_n,b_1,\cdots,b_n) \in \mathcal{C}$. Then, solutions of \eqref{eq: system main for original} coincide with minimums of $M$. Accordingly, we find the solution of \eqref{eq: system main for original} by numerically solving the optimization problem
\begin{align}\label{eq:minM}
\min\limits_{(\ba,\bb)\in \mathcal{C}} M(\ba,\bb).
\end{align}
We devise our algorithm in Wolfram Mathematica\textsuperscript{\textregistered} language and use the built-in optimization function \texttt{FindMinimum} to solve \eqref{eq:minM}.

As an example, we consider the kernel
\begin{align*}
G_1(x)=1+4 \cos(2\pi x)- 5 \sin(2\pi x)+\cos(4\pi x)-2 \sin(4\pi x),\quad x\in \Tt.
\end{align*}
We denote by $(\widetilde{u}_1,\widetilde{m}_1,\widetilde{H}_1)$ the corresponding numerical solution of \eqref{eq: main}. We first find $(\widetilde{m}_1,\widetilde{H}_1)$ by solving \eqref{eq:minM} and using \eqref{eq: ansatz m} and \eqref{eq: H=a0-p0}. Next, we use \eqref{eq: reduction} to find $\widetilde{u}_1$.

Finally, to estimate the accuracy of numerical solutions we introduce the error function
\begin{align*}
\mathrm{Er}_1(x)&=\left|\frac{j^2}{2\widetilde{m}_1(x)^{\alpha}}+V(x)-\int\limits_{\Tt} G_1(x-y)\widetilde{m}_1(y)dy-\widetilde{H}_1\right|\\ \nonumber
&+\left|\int\limits_{\Tt}\widetilde{m}_1(y)dy-1\right|,\ x\in \Tt.
\end{align*} 

We plot $G_1$ and $V$ in Fig. \ref{fig:G1andV}, $\widetilde{m}_1$ and $\widetilde{u}_1$ in Fig. \ref{fig:u1andm1} and $\mathrm{Er}_1$ in Fig. \ref{fig:Er1}.

\begin{figure}
\includegraphics[width=10cm]{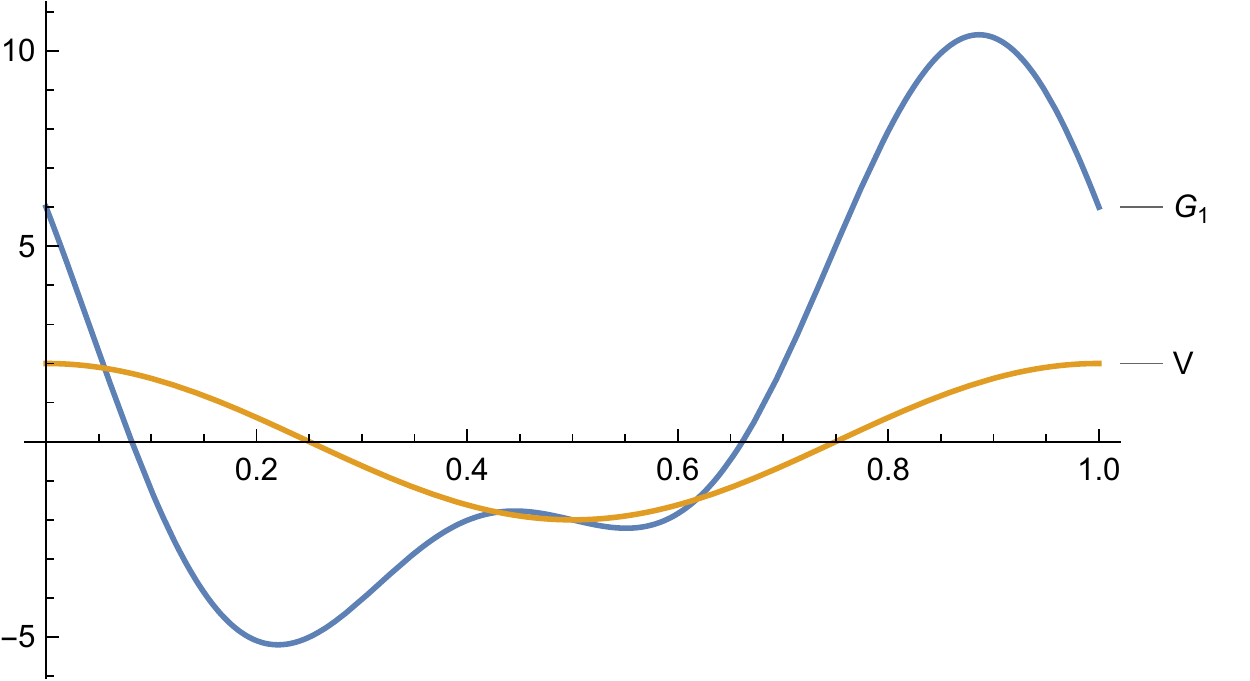}
\caption{The kernel $G_1$ and the potential $V$.}
\label{fig:G1andV}
\end{figure}

\begin{figure}
	\includegraphics[width=10cm]{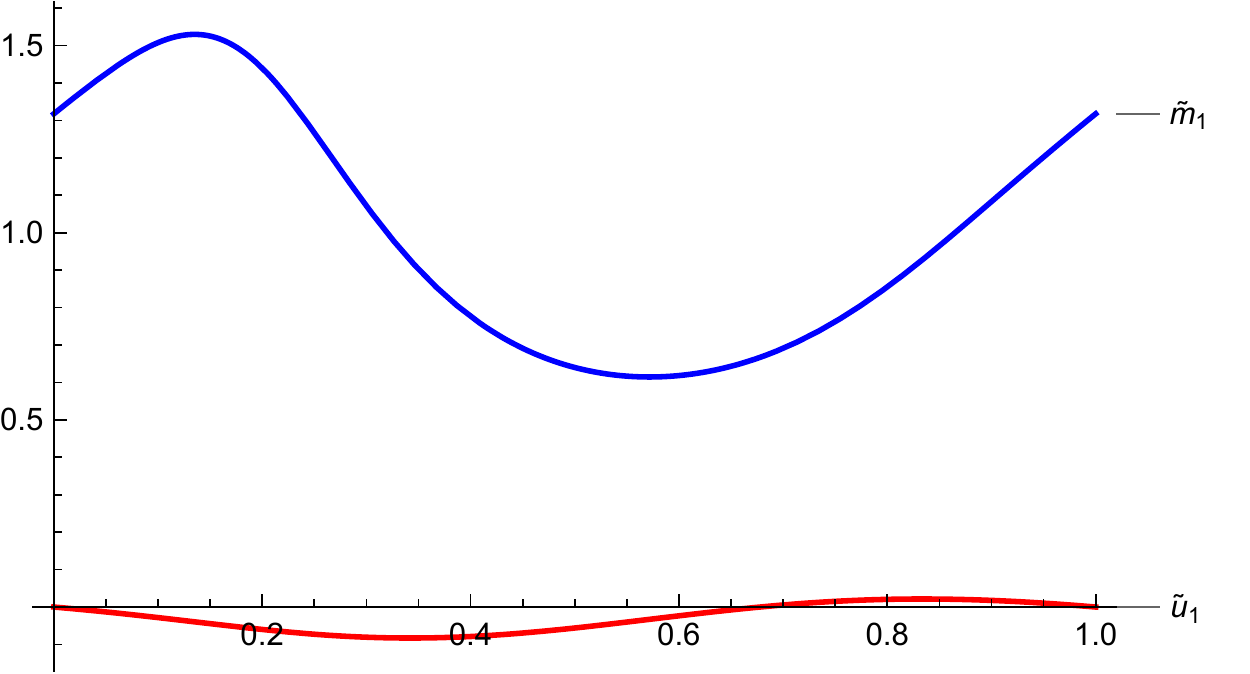}
\caption{The approximate solutions $\widetilde{m}_1$ and $\widetilde{u}_1$.}
\label{fig:u1andm1}
\end{figure}

\begin{figure}
\includegraphics[width=10cm]{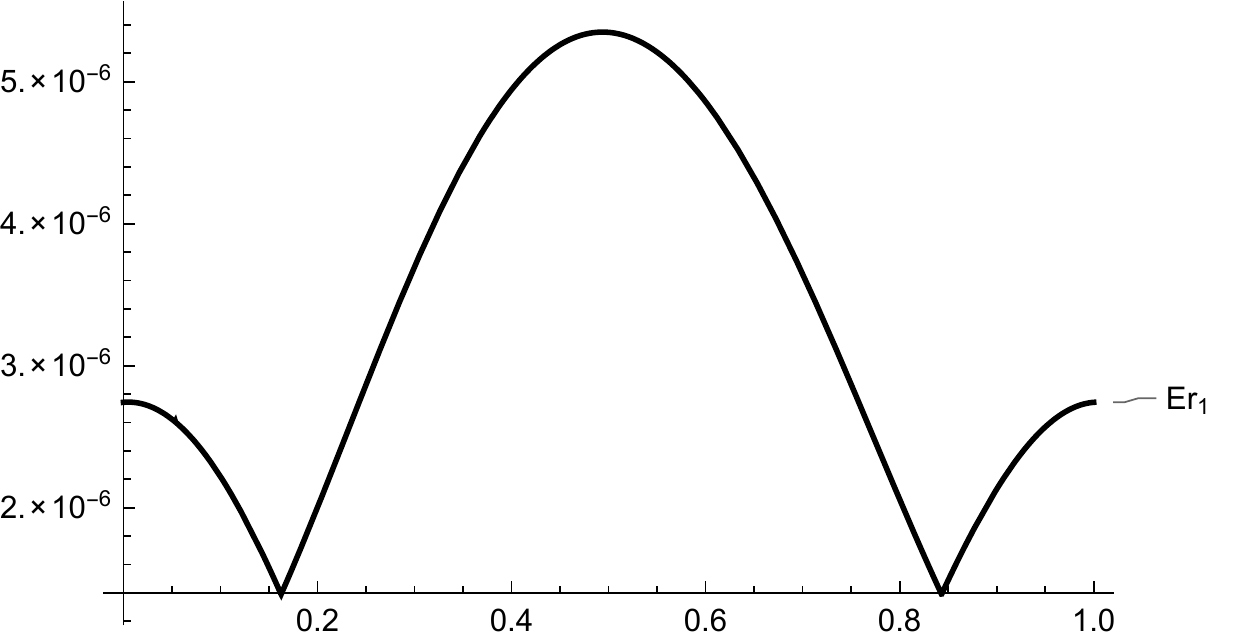}
\caption{The error $\mathrm{Er}_1$.}
\label{fig:Er1}
\end{figure}

\subsection{The case of a symmetric trigonometric polynomial}

By Corollary \ref{crl:sym_case} we have that for a given symmetric trigonometric polynomial $G$ the solution $m$ of \eqref{eq: main current} has the form \eqref{eq: ansatz m}, where the vector $(a_0^*,a_1^*,\cdots,a_n^*,b_1^*,\cdots,b_n^*)$ is the unique solution of \eqref{eq: variational}. As before, we use \texttt{FindMinimum} to solve \eqref{eq: variational}.

As an example, we consider the kernel
\begin{align*}
G_2(x)=1+4 \cos(2\pi x)+\cos(4\pi x)+5 \cos(6\pi x)+7 \cos (8 \pi x),\quad x\in \Tt.
\end{align*}
Analogous to the previous case we denote by $(\widetilde{u}_2,\widetilde{m}_2,\widetilde{H}_2)$ the numerical solution of \eqref{eq: main} corresponding to $G_2$. Furthermore, we denote by $\mathrm{Er}_2$ the error function corresponding to $(\widetilde{u}_2,\widetilde{m}_2,\widetilde{H}_2)$. We plot $G_2$ and $V$ in Fig. \ref{fig:G2andV}, $\widetilde{m}_2$ and $\widetilde{u}_2$ in Fig. \ref{fig:u2andm2} and $\mathrm{Er}_2$ in Fig. \ref{fig:Er2}.

\begin{figure}
\includegraphics[width=10cm]{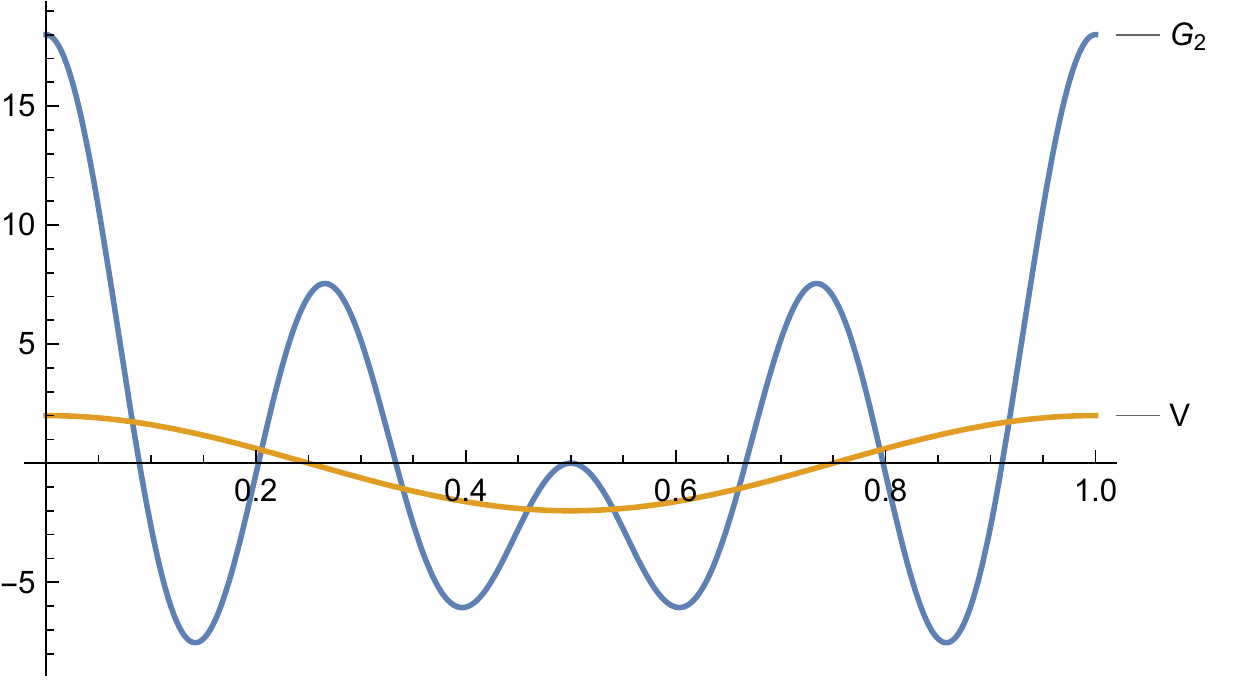}
\caption{The kernel $G_2$ and the potential $V$.}
\label{fig:G2andV}
\end{figure}

\begin{figure}
\includegraphics[width=10cm]{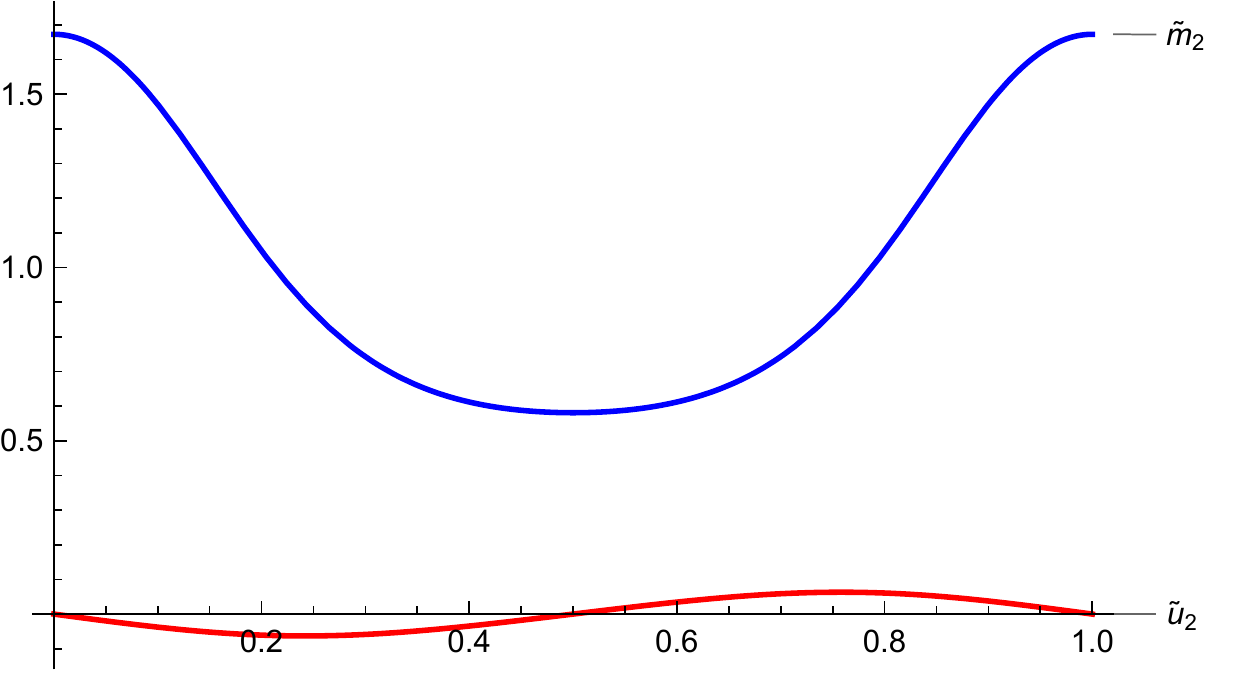}
\caption{The approximate solutions $\widetilde{m}_2$ and $\widetilde{u}_2$.}
\label{fig:u2andm2}
\end{figure}

\begin{figure}
\includegraphics[width=10cm]{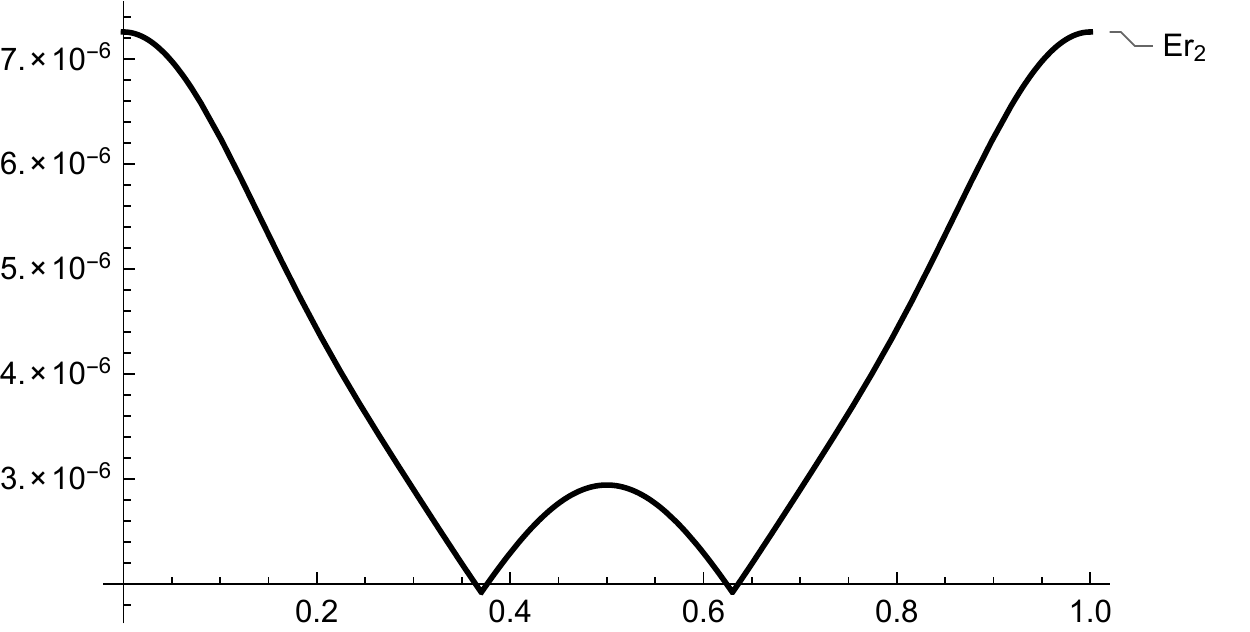}
\caption{The error $\mathrm{Er}_2$.}
\label{fig:Er2}
\end{figure}

\subsection{The case of a non trigonometric polynomial}

If $G$ is not a trigonometric polynomial we first approximate it by its truncated Fourier series and then apply one of the previous solution methods. As an example we take
\begin{align*}
G_3(x)=\frac{2-\cos(2\pi x)+\sin (2 \pi x)}{5-4\cos(2 \pi x)},\quad x\in \Tt.
\end{align*}
As before, we denote by $(\widetilde{u}_3,\widetilde{m}_3,\widetilde{H}_3)$ and $\mathrm{Er}_3$ the numerical solution of \eqref{eq: main} and the error function corresponding to $G_3$, respectively. We plot $G_3$ and $V$ in Fig. \ref{fig:G3andV}, $\widetilde{m}_3$ and $\widetilde{u}_3$ in Fig. \ref{fig:u3andm3} and $\mathrm{Er}_3$ in Fig. \ref{fig:Er3}.

\begin{figure}
\includegraphics[width=10cm]{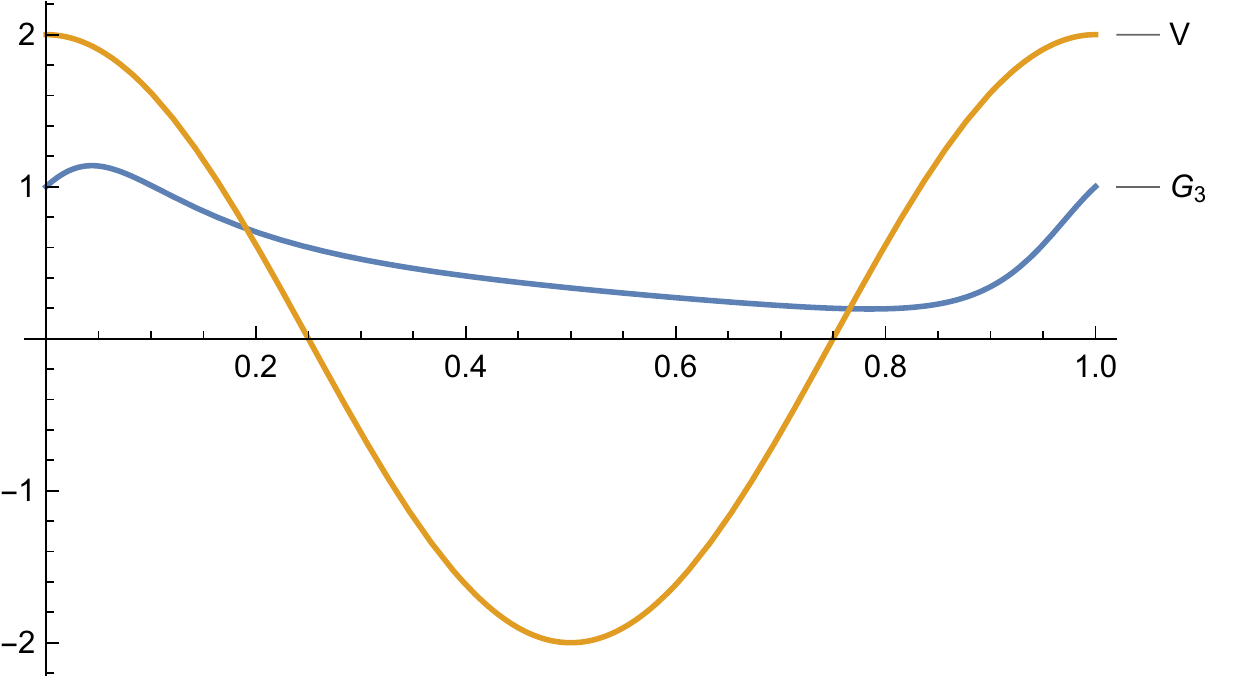}
\caption{The kernel $G_3$ and the potential $V$.}
\label{fig:G3andV}
\end{figure}

\begin{figure}
\includegraphics[width=10cm]{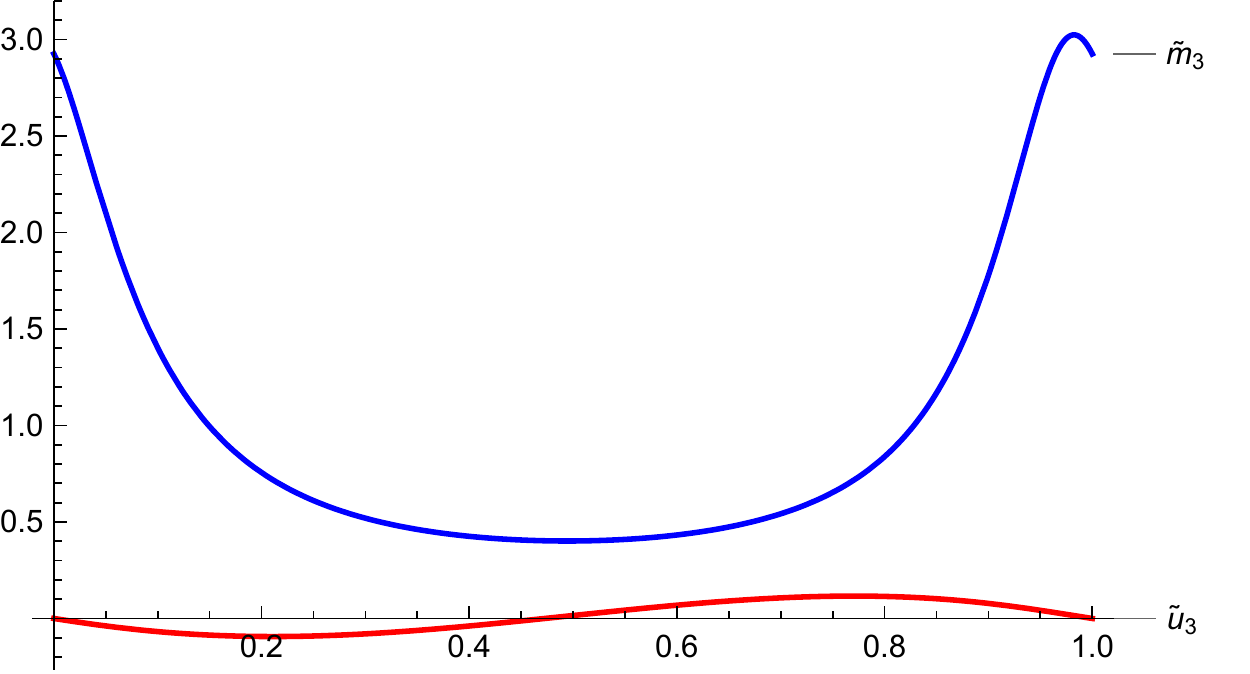}
\caption{The approximate solutions $\widetilde{m}_3$ and $\widetilde{u}_3$.}
\label{fig:u3andm3}
\end{figure}

\begin{figure}
\includegraphics[width=10cm]{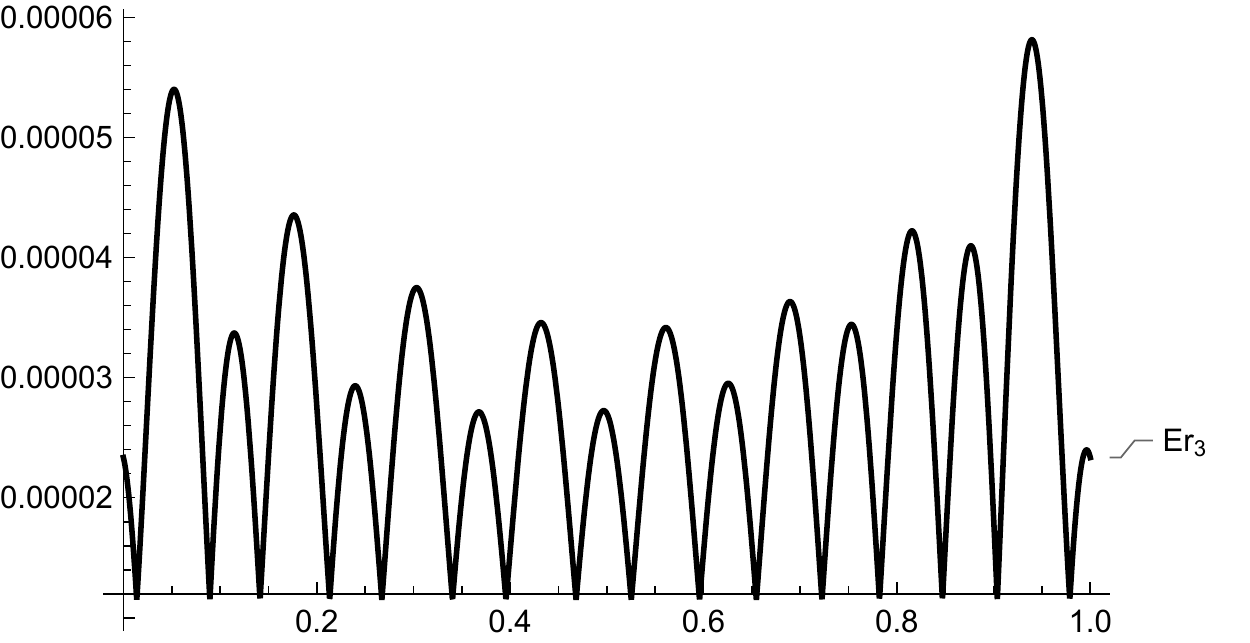}
\caption{The error $\mathrm{Er}_3$.}
\label{fig:Er3}
\end{figure}

\section{Extensions}\label{sec:extensions}

Here, we discuss how our methods can be applied to other one-dimensional MFG system such as \eqref{eq: main extension}. Denote by $L:\Tt\times \Rr\times\Rr^+ \to \Rr,\ (x,v,m)\mapsto L(x,v,m)$, be the Legendre transform of $H$; that is,
\begin{equation*}
	L(x,v,m)=\sup\limits_{p \in \Rr}\left(vp-H(x,p,m)\right).
\end{equation*}
Then, if $H$ satisfies suitable conditions, we have that
\begin{equation}\label{eq: Young}
	L(x,v,m)+H(x,p,m)\geq vp,
\end{equation}
for all $v,p \in \Rr$ and there is equality in \eqref{eq: Young} if and only if 
\begin{equation}\label{eq: Legendre transform}
	v=H'_p(x,p,m)\quad\text{or}\quad p=L'_v(x,v,m).
\end{equation}
As before, second equation in \eqref{eq: main extension} yields
\begin{equation*}
	H'_p(x,u_x,m)=\frac{j}{m},
\end{equation*}
for some constant $j$. Therefore, using \eqref{eq: Legendre transform} we find
\begin{equation*}
	u_x=L'_v\left(x,\frac{j}{m},m\right),
\end{equation*}
which we plug-in to the first equation in \eqref{eq: main extension} and obtain the following system
\begin{equation}\label{eq: main extension current}
	\begin{cases}
	H\left(x,L'_v\left(x,\frac{j}{m},m\right),m\right)=\mathcal{F}\left(\int\limits_{\Tt} G(x-y)m(y)dy\right)+\Hh,\\
	m>0,\ \int\limits_{\Tt} m(x)dx=1.
	\end{cases}
\end{equation}
Next, one can attempt to study \eqref{eq: main extension current} first when $G$ is a trigonometric polynomial and then approximate the general case. As before, when $G$ is a trigonometric polynomial the expression
\begin{equation*}
	\int\limits_{\Tt} G(x-y)m(y)dy
\end{equation*}
is always a trigonometric polynomial. Therefore, we have that
\begin{equation}\label{eq: ansatz extension}
H\left(x,L'_v\left(x,\frac{j}{m},m\right),m\right)=\mathcal{F}\left(\sum\limits_{k=0}^{n} a_k^* \cos(2 \pi kx)+b_k^*\sin(2 \pi k x)\right)+\Hh,
\end{equation}
for some $\{a_k^*\},\{b_k^*\} \subset \Rr$. Suppose $H$ is such that the left-hand-side expression of \eqref{eq: ansatz extension} is invertible in $m$ with inverse $A_j(x,m)$. Then, \eqref{eq: ansatz extension} yields the following ansatz
\begin{equation}\label{eq: ansatz m extension}
	m(x)=A_j\left(x,\mathcal{F}\left(\sum\limits_{k=0}^{n} a_k^* \cos(2 \pi kx)+b_k^*\sin(2 \pi k x)\right)+\Hh\right).
\end{equation}
Thus, one can search for the solution $m$ of \eqref{eq: main extension current} in the form \eqref{eq: ansatz m extension} with undetermined coefficients $\{a_k^*\},\{b_k^*\},\Hh$. Therefore, by plugging \eqref{eq: ansatz m extension} in \eqref{eq: main extension current} we obtain a finite-dimensional fixed point problem for $\{a_k^*\},\{b_k^*\},\Hh$. If this fixed point problem has good structural properties (such as \eqref{eq: system main for original}) for a concrete model of the form \eqref{eq: main extension}, one may analyze this model by methods developed here.


\begin{thebibliography}{10}
	
	\bibitem{achdou2013finite}
	Y.~Achdou.
	\newblock Finite difference methods for mean field games.
	\newblock In {\em Hamilton-Jacobi Equations: Approximations, Numerical Analysis
		and Applications}, pages 1--47. Springer, 2013.
	
	\bibitem{CDY}
	Y.~Achdou, F.~Camilli, and I.~Capuzzo-Dolcetta.
	\newblock Mean field games: numerical methods for the planning problem.
	\newblock {\em SIAM J. Control Optim.}, 50(1):77--109, 2012.
	
	\bibitem{DY}
	Y.~Achdou and I.~Capuzzo-Dolcetta.
	\newblock Mean field games: numerical methods.
	\newblock {\em SIAM J. Numer. Anal.}, 48(3):1136--1162, 2010.
	
	\bibitem{achdouperez'12}
	Y.~Achdou and V.~Perez.
	\newblock Iterative strategies for solving linearized discrete mean field games
	systems.
	\newblock {\em Netw. Heterog. Media}, 7(2):197--217, 2012.
	
	\bibitem{AP}
	Y.~Achdou and A.~Porretta.
	\newblock Convergence of a finite difference scheme to weak solutions of the
	system of partial differential equations arising in mean field games.
	\newblock {\em SIAM J. Numer. Anal.}, 54(1):161--186, 2016.
	
	\bibitem{AFG}
	N.~Al-Mulla, R.~Ferreira, and D.~Gomes.
	\newblock Two numerical approaches to stationary mean-field games.
	\newblock {\em Preprint}.
	
	\bibitem{camillisilva'12}
	F.~Camilli and F.~Silva.
	\newblock A semi-discrete approximation for a first order mean field game
	problem.
	\newblock {\em Netw. Heterog. Media}, 7(2):263--277, 2012.
	
	\bibitem{Carda}
	P.~Cardaliaguet.
	\newblock {\em Notes on mean field games}.
	\newblock 2013.
	
	\bibitem{Cd2}
	P.~Cardaliaguet.
	\newblock Weak solutions for first order mean field games with local coupling.
	\newblock In {\em Analysis and geometry in control theory and its
		applications}, volume~11 of {\em Springer INdAM Ser.}, pages 111--158.
	Springer, Cham, 2015.
	
	\bibitem{cgbt}
	P.~Cardaliaguet, P.~Garber, A.~Porretta, and D.~Tonon.
	\newblock Second order mean field games with degenerate diffusion and local
	coupling.
	\newblock {\em Preprint}, 2014.
	
	\bibitem{meszaros'16}
	P.~Cardaliaguet, A.~R. M{\'e}sz{\'a}ros, and F.~Santambrogio.
	\newblock First {O}rder {M}ean {F}ield {G}ames with {D}ensity {C}onstraints:
	{P}ressure {E}quals {P}rice.
	\newblock {\em SIAM J. Control Optim.}, 54(5):2672--2709, 2016.
	
	\bibitem{carlinisilva'14}
	E.~Carlini and F.~J. Silva.
	\newblock A fully discrete semi-{L}agrangian scheme for a first order mean
	field game problem.
	\newblock {\em SIAM J. Numer. Anal.}, 52(1):45--67, 2014.
	
	\bibitem{cirant'16}
	M.~Cirant.
	\newblock Stationary focusing mean-field games.
	\newblock {\em Comm. Partial Differential Equations}, 41(8):1324--1346, 2016.
	
	\bibitem{javier'01}
	J.~Duoandikoetxea.
	\newblock {\em Fourier analysis}, volume~29 of {\em Graduate Studies in
		Mathematics}.
	\newblock American Mathematical Society, Providence, RI, 2001.
	\newblock Translated and revised from the 1995 Spanish original by David
	Cruz-Uribe.
	
	\bibitem{FG2}
	R.~Ferreira and D.~Gomes.
	\newblock Existence of weak solutions for stationary mean-field games through
	variational inequalities.
	\newblock {\em Preprint}.
	
	\bibitem{GMit}
	D.~Gomes and H.~Mitake.
	\newblock Existence for stationary mean-field games with congestion and
	quadratic {H}amiltonians.
	\newblock {\em NoDEA Nonlinear Differential Equations Appl.}, 22(6):1897--1910,
	2015.
	
	\bibitem{GNP2}
	D.~Gomes, L.~Nurbekyan, and M.~Prazeres.
	\newblock Explicit solutions of one-dimensional, first-order, stationary
	mean-field games with congestion.
	\newblock {\em 2016 IEEE 55th Conference on Decision and Control (CDC)}, 2016.
	
	\bibitem{GNP}
	D.~Gomes, L.~Nurbekyan, and M.~Prazeres.
	\newblock Explicit solutions of one-dimensional stationary mean-field games
	with a local coupling.
	\newblock {\em Preprint}, 2016.
	
	\bibitem{GPim2}
	D.~Gomes and E.~Pimentel.
	\newblock Time dependent mean-field games with logarithmic nonlinearities.
	\newblock {\em To appear in SIAM Journal on Mathematical Analysis}.
	
	\bibitem{GPim1}
	D.~Gomes and E.~Pimentel.
	\newblock Local regularity for mean-field games in the whole space.
	\newblock {\em To appear in Minimax Theory and its Applications}, 2015.
	
	\bibitem{GPM3}
	D.~Gomes, E.~Pimentel, and H.~S\'anchez-Morgado.
	\newblock Time dependent mean-field games in the superquadratic case.
	\newblock {\em To appear in ESAIM: Control, Optimisation and Calculus of
		Variations}.
	
	\bibitem{GPM2}
	D.~Gomes, E.~Pimentel, and H.~S{\'a}nchez-Morgado.
	\newblock Time-dependent mean-field games in the subquadratic case.
	\newblock {\em Comm. Partial Differential Equations}, 40(1):40--76, 2015.
	
	\bibitem{GPV}
	D.~Gomes, E.~Pimentel, and V.~Voskanyan.
	\newblock {\em Regularity theory for mean-field game systems}.
	\newblock 2016.
	
	\bibitem{GPM1}
	D.~Gomes, G.~E. Pires, and H.~S{\'a}nchez-Morgado.
	\newblock A-priori estimates for stationary mean-field games.
	\newblock {\em Netw. Heterog. Media}, 7(2):303--314, 2012.
	
	\bibitem{GM}
	D.~Gomes and H.~S{\'a}nchez-Morgado.
	\newblock A stochastic {E}vans-{A}ronsson problem.
	\newblock {\em Trans. Amer. Math. Soc.}, 366(2):903--929, 2014.
	
	\bibitem{GS}
	D.~Gomes and J.~Sa{\'u}de.
	\newblock Mean field games models---a brief survey.
	\newblock {\em Dyn. Games Appl.}, 4(2):110--154, 2014.
	
	\bibitem{GVrt2}
	D.~Gomes and V.~Voskanyan.
	\newblock Short-time existence of solutions for mean-field games with
	congestion.
	\newblock {\em J. Lond. Math. Soc. (2)}, 92(3):778--799, 2015.
	
	\bibitem{Graber2}
	J.~Graber.
	\newblock Weak solutions for mean field games with congestion.
	\newblock {\em Preprint}, 2015.
	
	\bibitem{gueant'12}
	O.~Gu{\'e}ant.
	\newblock New numerical methods for mean field games with quadratic costs.
	\newblock {\em Netw. Heterog. Media}, 7(2):315--336, 2012.
	
	\bibitem{Gueant2}
	O.~Gu{\'e}ant.
	\newblock Existence and {U}niqueness {R}esult for {M}ean {F}ield {G}ames with
	{C}ongestion {E}ffect on {G}raphs.
	\newblock {\em Appl. Math. Optim.}, 72(2):291--303, 2015.
	
	\bibitem{Caines2}
	M.~Huang, P.~E. Caines, and R.~P. Malham{\'e}.
	\newblock Large-population cost-coupled {LQG} problems with nonuniform agents:
	individual-mass behavior and decentralized {$\epsilon$}-{N}ash equilibria.
	\newblock {\em IEEE Trans. Automat. Control}, 52(9):1560--1571, 2007.
	
	\bibitem{Caines1}
	M.~Huang, R.~P. Malham{\'e}, and P.~E. Caines.
	\newblock Large population stochastic dynamic games: closed-loop
	{M}c{K}ean-{V}lasov systems and the {N}ash certainty equivalence principle.
	\newblock {\em Commun. Inf. Syst.}, 6(3):221--251, 2006.
	
	\bibitem{ll1}
	J.-M. Lasry and P.-L. Lions.
	\newblock Jeux \`a champ moyen. {I}. {L}e cas stationnaire.
	\newblock {\em C. R. Math. Acad. Sci. Paris}, 343(9):619--625, 2006.
	
	\bibitem{ll2}
	J.-M. Lasry and P.-L. Lions.
	\newblock Jeux \`a champ moyen. {II}. {H}orizon fini et contr\^ole optimal.
	\newblock {\em C. R. Math. Acad. Sci. Paris}, 343(10):679--684, 2006.
	
	\bibitem{ll3}
	J.-M. Lasry and P.-L. Lions.
	\newblock Mean field games.
	\newblock {\em Jpn. J. Math.}, 2(1):229--260, 2007.
	
	\bibitem{ll4}
	J.-M. Lasry and P.-L. Lions.
	\newblock Mean field games.
	\newblock {\em Cahiers de la Chaire Finance et D\'eveloppement Durable}, 2007.
	
	\bibitem{llg2}
	J.-M. Lasry, P.-L. Lions, and O.~Gu{\'e}ant.
	\newblock Mean field games and applications.
	\newblock {\em Paris-Princeton lectures on Mathematical Finance}, 2010.
	
	\bibitem{LCDF}
	P.-L. Lions.
	\newblock College de france course on mean-field games.
	\newblock 2007-2011.
	
	\bibitem{meszaros'15}
	A.~R. M{\'e}sz{\'a}ros and F.~J. Silva.
	\newblock A variational approach to second order mean field games with density
	constraints: the stationary case.
	\newblock {\em J. Math. Pures Appl. (9)}, 104(6):1135--1159, 2015.
	
	\bibitem{PV15}
	E.~Pimentel and V.~Voskanyan.
	\newblock Regularity for second-order stationaty mean-field games.
	\newblock {\em To appear in Indiana University Mathematics Journal}.
	
	\bibitem{porretta}
	A.~Porretta.
	\newblock On the planning problem for the mean field games system.
	\newblock {\em Dyn. Games Appl.}, 4(2):231--256, 2014.
	
	\bibitem{porretta2}
	A.~Porretta.
	\newblock Weak solutions to {F}okker-{P}lanck equations and mean field games.
	\newblock {\em Arch. Ration. Mech. Anal.}, 216(1):1--62, 2015.
	
\end{thebibliography}

\def\polhk#1{\setbox0=\hbox{#1}{\ooalign{\hidewidth
			\lower1.5ex\hbox{`}\hidewidth\crcr\unhbox0}}} \def\cprime{$'$}

\end{document}